\documentclass{article}%
\usepackage{amsmath}
\usepackage{amsfonts}
\usepackage{amssymb}
\usepackage{graphicx}%
\providecommand{\U}[1]{\protect \rule{.1in}{.1in}}
\newtheorem{theorem}{Theorem}

\newtheorem{corollary}[theorem]{Corollary}

\newtheorem{example}[theorem]{Example}

\newtheorem{lemma}[theorem]{Lemma}

\newtheorem{proposition}[theorem]{Proposition}

\newenvironment{proof}[1][Proof]{\noindent \textbf{#1.} }{\  \rule{0.5em}{0.5em}}
\begin{document}

\title{Approximate Double Commutants and Distance Formulas}
\author{Don Hadwin\\University of New Hampshire
\and Junhao Shen\\University of New Hampshire}
\maketitle

\begin{abstract}
We extend work of the first author concering relative double commutants and
approximate double commutants of unital subalgebras of unital C*-algebras,
including metric versions involving distance estimates. We prove metric
results for AH subalgebras of von Neumann algebras or AF subalgebras of
primitive C*-algebras. We prove other general results, including some for
nonselfadjoint commutative subalgebras, using C*-algebraic versions of the
Stone-Weierstrass and Bishop-Stone-Weierstrass theorems.

\end{abstract}

\section{Introduction}

We extend results in \cite{H4} on approximate double commutants of
C*-subalgebras of C*-algebras. We also obtain some results for non-selfadjoint
subalgebras. A key ingredient in the proof of the main result in \cite{H4} was
S. Machado's version \cite{M} of the Bishop-Stone-Weierstrass theorem
\cite{B}. In this paper we use Machado's vector version of his theorem
\cite{M}, the factor state version of the Stone-Weierstrass theorem for
C*-algebras of R. Longo \cite{L}, S. Popa \cite{P}, S. Teleman \cite{T}, and
the first author's version of the Bishop-Stone-Weierstrass theorem for
C*-algebras \cite{H3}.

The classical double commutant theorem of von Neumann \cite{vN} is a key
result in the theory of von Neumann algebras. The first author \cite{H1}
proved an asymptotic version of von Neumann's result for unital C*-algebras,
and later \cite{H2} proved a metric version with an analogue of Arveson's
distance formula.

It was shown by R. Kadison \cite{K} that inside a factor von Neumann algebra
von Neumann's double commutant theorem fails, even for commutative
subalgebras. However, the author \cite{H4} showed that, for commutative unital
C*-subalgebras of a factor von Neumann algebra, the asymptotic version holds.

Suppose $\mathcal{S}$ is a subset of a ring $\mathcal{R}$. We define the
\emph{relative commutant} of $\mathcal{S}$ in $\mathcal{R}$, the
\emph{relative double commutant} of $\mathcal{S}$ in $\mathcal{R}$,
respectively, by%
\[
\left(  \mathcal{S},\mathcal{R}\right)  ^{\prime}=\left \{  T\in \mathcal{R}%
:\forall S\in \mathcal{S},TS=ST\right \}  ,
\]%
\[
\left(  \mathcal{S},\mathcal{R}\right)  ^{\prime \prime}=\left \{
T\in \mathcal{R}:\forall A\in \left(  \mathcal{S},\mathcal{R}\right)  ^{\prime
},TA=AT\right \}  .
\]
If $\mathcal{B}$ is a unital C*-algebra and $\mathcal{S}\subseteq \mathcal{B}$,
we define the \emph{relative approximate double commutant of }$\mathcal{S}%
$\emph{ in }$\mathcal{B}$, denoted by Appr$\left(  \mathcal{S},\mathcal{B}%
\right)  ^{\prime \prime}$ as the set of all $T\in \mathcal{B}$ such that
\[
\left \Vert TA_{\lambda}-A_{\lambda}T\right \Vert \rightarrow0
\]
for every bounded net $\left \{  A_{\lambda}\right \}  $ in $\mathcal{B}$ for
which%
\[
\left \Vert SA_{\lambda}-A_{\lambda}S\right \Vert \rightarrow0
\]
for every $S\in \mathcal{S}$. The approximate double commutant theorem
\cite{H1} in $B\left(  H\right)  $ says that if $\mathcal{S=S}^{\ast},$ then
Appr$\left(  \mathcal{S}\right)  ^{\prime \prime}=C^{\ast}\left(
\mathcal{S}\right)  $. Moreover, if we restrict the $\left \{  A_{\lambda
}\right \}  $'s to be nets of unitaries or nets of projections that
asymptotically commute with every element of $\mathcal{S}$, the resulting
approximate double commutant is still $C^{\ast}\left(  \mathcal{S}\right)  $.

It is clear that the \emph{center} $\mathcal{Z}\left(  \mathcal{B}\right)  $
of $\mathcal{B}$ is always contained in Appr$\left(  \mathcal{S}%
,\mathcal{B}\right)  ^{\prime \prime}$ and that Appr$\left(  \mathcal{S}%
,\mathcal{B}\right)  ^{\prime \prime}$ is a norm closed unital algebra. Thus
Appr$\left(  \mathcal{S},\mathcal{B}\right)  ^{\prime \prime}$ always contains
the norm closed unital algebra generated by $\mathcal{S\cup Z}\left(
\mathcal{B}\right)  $. If $\mathcal{S=S}^{\ast}$, then Appr$\left(
\mathcal{S},\mathcal{B}\right)  ^{\prime \prime}$ is a C*-algebra and must
contain $C^{\ast}\left(  \mathcal{S\cup Z}\left(  \mathcal{B}\right)  \right)
$. In \cite{K} R. Kadison calls a subalgebra $\mathcal{A}$ of $\mathcal{B}$
\emph{normal} if $\mathcal{A=}$ $\left(  \mathcal{A},\mathcal{B}\right)
^{\prime \prime}$. We say that $\mathcal{A}$ is \emph{approximately normal} if
$\mathcal{A=}$ $Appr\left(  \mathcal{A},\mathcal{B}\right)  ^{\prime \prime}$.

We say that $\mathcal{A}$ is \emph{metric-normal }in $\mathcal{B}$ if there is
a constant $K<\infty$ such that, for every $T\in \mathcal{B}$,%
\[
dist\left(  T,\mathcal{A}\right)  \leq K\sup \left \{  \left \Vert
TU-UT\right \Vert :U\in \left(  \mathcal{A},\mathcal{B}\right)  ^{\prime}\text{,
}U\text{ unitary}\right \}  .
\]
The smallest such $K$ is the \emph{constant of metric-normality} $K_{n}\left(
\mathcal{A},\mathcal{B}\right)  $ of $\mathcal{A}$ in $\mathcal{B}$. We say
that $\mathcal{A}$ is \emph{approximately metric-normal} in $\mathcal{B}$ if
there is a $K<\infty$ such that, for every $T\in \mathcal{B}$ there is a net
$\left \{  U_{\lambda}\right \}  $ of unitaries in $\mathcal{B}$ such that, for
every $A\in \mathcal{A}$, $\left \Vert AU_{\lambda}-U_{\lambda}A\right \Vert
\rightarrow0$, and such that%
\[
dist\left(  T,\mathcal{A}\right)  \leq K\lim_{\lambda}\left \Vert TU_{\lambda
}-U_{\lambda}T\right \Vert .
\]
The smallest such $K$ is the \emph{constant of approximate metric normality}
$K_{an}\left(  \mathcal{A},\mathcal{B}\right)  $.

Here is a summary of the results in this paper. In Section 2 we discuss a
version of relative injectivity, summarize known results and prove a few new
ones. We relate the forms of injectivity to the metric versions of normalilty
and approximate normality. We also develop a number of useful basic results
about the various versions of normality. We prove (Theorem \ref{AH}) that if
$\mathcal{A}$ is a unital AH C*-subalgebra $\mathcal{A}$ of a von Neumann
algebra $\mathcal{B}$, then $C^{\ast}\left(  \mathcal{A\cup Z}\left(
\mathcal{B}\right)  \right)  $ is metric approximately normal in $\mathcal{B}%
$, and we prove (Theorem \ref{AF}) that every unital AF C*-subalgebra of a
primitive C*-algebra is metric approximately normal.

In Section 3, following ideas of Akemann and Pedersen \cite{AP}, we prove
(Theorem \ref{lift}) that surjective unital $\ast$-homomormisms send the
approximate double commutant of a set into the approximate double commutant of
the image of the set. This result is a key ingredient to our results in
Sections 4 and 5 where we prove general results (Theorem \ref{SWapp} and
Theorem) that involve C*-algebraic versions of the Stone-Weierstrass or
Bishop-Stone-Weierstrass theorems. We conclude in Section 6 with a list of
open problems.

\section{Metric Results}

If $\mathcal{A}$ is a unital C*-subalgebra of a C*-algebra $\mathcal{B}$, then
$\mathcal{F}\left(  \mathcal{A},\mathcal{B}\right)  $ is the convex hull of
the maps $Ad_{U}:\mathcal{B}\rightarrow \mathcal{B}$ defined by $Ad_{u}\left(
T\right)  =U^{\ast}TU$, with a unitary $U\in \left(  \mathcal{A},\mathcal{B}%
\right)  ^{\prime}$. We say that a unital C*-subalgebra $\mathcal{A}$ is
\emph{strongly injective} in a unital C*-algebra $\mathcal{B}$ if there is a
conditional expectation $E:\mathcal{B}\rightarrow \mathcal{A}$, a faithful
unital representation $\pi:\mathcal{B}\rightarrow B\left(  H\right)  $ for
some Hilbert space $H$, and a net $\left \{  \varphi_{\lambda}\right \}  $ in
$\mathcal{F}\left(  \mathcal{A},\mathcal{B}\right)  $ such that, for every
$T\in \mathcal{B}$,%
\[
\pi \left(  \varphi_{\lambda}\left(  T\right)  \right)  \rightarrow \pi \left(
E\left(  T\right)  \right)
\]
in the weak operator topology. It is clear that if $\mathcal{A}$ is strongly
injective in $\mathcal{B}$, then $\mathcal{A}$ contains the center
$\mathcal{Z}\left(  \mathcal{B}\right)  $ of $\mathcal{B}$. If $\mathcal{A}$
and $\mathcal{B}$ are von Neumann algebras, we say that $\mathcal{A}$ is
\emph{weak* injective} in $\mathcal{B}$ if $E$ and the net $\left \{
\varphi_{\lambda}\right \}  $ and be chosen so that, for every $T\in
\mathcal{B}$,%
\[
\varphi_{\lambda}\left(  T\right)  \rightarrow E\left(  T\right)
\]
in the weak*-topology on $\mathcal{B}$, i.e., we can choose $\pi$ to be the
identity representation on $\mathcal{B}$.

\begin{proposition}
\label{inj}Suppose $\mathcal{B}\subseteq B\left(  H\right)  $ is a unital
C*-algebra. Then

\begin{enumerate}
\item If $\pi:\mathcal{B}\rightarrow B\left(  M\right)  $ is a faithful unital
$\ast$-homomorphism for some Hilbert space $M$, and $E:\mathcal{B}%
\rightarrow \mathcal{A}$ is a conditional expectation and there is a net
$\left \{  \psi_{\lambda}\right \}  $ in $\mathcal{F}\left(  \pi \left(
\mathcal{A}\right)  ,\pi \left(  \mathcal{B}\right)  ^{\prime \prime}\right)  $
such that, for every $T\in \mathcal{B}$, $\left \{  \psi_{\lambda}\left(
\pi \left(  T\right)  \right)  \right \}  $ converges in the weak operator
topology to an element $\pi \left(  E\left(  T\right)  \right)  $ of
$\pi \left(  \mathcal{A}\right)  $, then $\mathcal{A}$ is strongly injective in
$\mathcal{B}$.

\item \cite[Theorem C]{GK}If $\mathcal{B}$ is a von Neumann algebra, then
$\mathcal{Z}\left(  \mathcal{B}\right)  $ is weak* injective in $\mathcal{B}$.

\item \cite{R} If $\mathcal{B}$ is a von Neumann algebra and $\mathcal{A}$ is
a normal von Neumann subalgebra of $\mathcal{B}$ such that $\left(
\mathcal{A},\mathcal{B}\right)  ^{\prime}$ is hyperfinite (e.g., $\mathcal{A}$
is a masa in $\mathcal{B}$), then $\mathcal{A}$ is weak* injective in
$\mathcal{B}$.

\item If $\mathcal{B}$ is a primitive C*-algebra, then $\mathcal{Z}\left(
\mathcal{B}\right)  =$ $\mathbb{C}1$ is strongly injective in $\mathcal{B}$.

\item If $\mathcal{A}=\sum_{1\leq j\leq m}^{\oplus}\mathcal{A}_{j}$ is a
unital C*-subalgebra of $\mathcal{B}$ and, for $i=1,\ldots,m$, $P_{1}%
=1\oplus0\oplus \cdots \oplus0$, $P_{2}=0\oplus1\oplus \cdots \oplus0$,$\ldots$,
$P_{m}=0\oplus \cdots \oplus0\oplus1,$ then $\mathcal{A}$ is strongly injective
(resp., normal, approximately normal) in $\mathcal{B}$ if $\mathcal{A}_{i}$ is
strongly injective (resp., normal,approximately normal) in $P_{i}%
\mathcal{B}P_{i}$ for $1\leq i\leq m$.

\item If $\mathcal{A}_{i}$ is strongly injective in $\mathcal{B}_{i}$ for
$i=1,2$, then $\mathcal{A}_{1}\otimes_{\text{\textrm{min}}}\mathcal{A}_{2}$ is
strongly injective in $\mathcal{B}_{1}\otimes_{\text{\textrm{min}}}%
\mathcal{B}_{2}$.

\item If$\mathcal{\ A}$ is strongly injective in $\mathcal{B}$ and
$\mathcal{W}$ is any unital C*-algebra, then $\mathcal{W}\otimes
_{\text{\textrm{min}}}\mathcal{A}$ is strongly injective in $\mathcal{W}%
\otimes_{\text{\textrm{min}}}\mathcal{B}$.

\item if $\mathcal{B}=\mathcal{M}_{k}\left(  \mathcal{D}\right)
=\mathcal{M}_{k}\left(  \mathbb{C}\right)  \otimes \mathcal{D}$ for
$k\in \mathbb{N}$ and $\mathcal{A}=\mathcal{M}_{k}\left(  \mathcal{E}\right)
=\mathcal{M}_{k}\left(  \mathbb{C}\right)  \otimes \mathcal{E}$ and
$\mathcal{E}$ is strongly injective (resp., normal) in $\mathcal{D}$, then
$\mathcal{A}$ is strongly injective (resp., normal) in $\mathcal{B}$.
\end{enumerate}
\end{proposition}

\begin{proof}
$\left(  1\right)  .$ Suppose $U\in \pi \left(  \mathcal{B}\right)
^{\prime \prime}$ is unitary. It follows that there is an $A=A^{\ast}\in
\pi \left(  \mathcal{B}\right)  ^{\prime \prime}$ such that $U=e^{iA}$. It
follows from the Kaplansky density theorem that there is a bounded net
$\left \{  A_{m}\right \}  $ in $\mathcal{B}$ such that $\pi \left(
A_{m}\right)  \rightarrow A$ in the strong operator topology, and it follows,
that if $U_{m}=e^{iA_{m}},$ then $\pi \left(  U_{m}\right)  \rightarrow U$ in
the $\ast$-strong operator topology, and thus $\pi \left(  Ad_{U_{m}}\left(
B\right)  \right)  \rightarrow Ad_{U}\left(  \pi \left(  B\right)  \right)  $
in the strong operator topology. It follows that the point-weak-operator
closure of $\left \{  \pi \circ \varphi:\varphi \in \mathcal{F}\left(
\mathcal{A},\mathcal{B}\right)  \right \}  $ contains every $\psi_{\lambda
}\circ \pi$, and thus contains $\pi \circ E$.It follows that $\mathcal{A}$ is
strongly injective in $\mathcal{B}$.

$\left(  4\right)  .$ This follows from $\left(  1\right)  $ and $\left(
2\right)  $.

$\left(  5\right)  $ is obvious.

$\left(  6\right)  .$ Suppose, for $i\in \left \{  1,2\right \}  $, $\pi
_{i}:\mathcal{B}_{i}\rightarrow B\left(  H_{i}\right)  $ is a faithful
representation, $E_{i}:\mathcal{B}_{i}\rightarrow \mathcal{A}_{i}$ is a
conditional expectation and $\left \{  \varphi_{\lambda,i}\right \}  $ is a net
in $\mathcal{F}\left(  \mathcal{A}_{i},\mathcal{B}_{i}\right)  $ such that
\[
\pi_{i}\left(  \varphi_{\lambda,i}\left(  T\right)  \right)  \rightarrow
\pi_{i}\left(  E_{i}\left(  T\right)  \right)
\]
in the weak operator topology. Then $E\left(  T_{1}\otimes T_{2}\right)
=E_{1}\left(  T_{1}\right)  \otimes E_{2}\left(  T_{2}\right)  $ defines a
conditional expectation $E:\mathcal{B}_{1}\otimes_{\text{\textrm{min}}%
}\mathcal{B}_{2}\rightarrow \mathcal{A}_{1}\otimes_{\text{\textrm{min}}%
}\mathcal{A}_{2}$. Also $\pi \left(  T_{1}\otimes T_{2}\right)  =\pi_{1}\left(
T_{1}\right)  \otimes \pi_{2}\left(  T_{2}\right)  \in B\left(  H_{1}\otimes
H_{2}\right)  $ defines a faithful representation of $\mathcal{B}_{1}%
\otimes_{\text{\textrm{min}}}\mathcal{B}_{2}$. Moreover, if $U_{k,i}\in \left(
\mathcal{A}_{i},\mathcal{B}_{i}\right)  ^{\prime}$ is unitary for $i=1,2$ and
$1\leq k\leq m$ and if $0\leq s_{1},t_{1},\ldots,s_{m},t_{m}$ and $\sum
_{k=1}^{m}s_{k}=\sum_{k=1}^{m}t_{k}=1$, then $W_{k,r}=U_{k,1}\otimes
U_{r,2}\in \left(  \mathcal{A}_{1}\otimes_{\text{\textrm{min}}}\mathcal{A}%
_{2},\mathcal{B}_{1}\otimes_{\text{\textrm{min}}}\mathcal{B}_{2}\right)
^{\prime}$ and%
\[
\sum_{k,r=1}^{m}s_{k}t_{r}W_{k,r}\left(  T_{1}\otimes T_{2}\right)
W_{k,r}^{\ast}=\left(  \sum_{k=1}^{m}s_{k}U_{k,1}T_{1}U_{k,1}^{\ast}\right)
\otimes \left(  \sum_{r=1}^{m}t_{r}U_{r,2}T_{2}U_{r,2}^{\ast}\right)  \text{.}%
\]
Thus%
\[
\varphi_{\lambda}\left(  T_{1}\otimes T_{2}\right)  =\varphi_{\lambda
,1}\left(  T_{1}\right)  \otimes \varphi_{\lambda,2}\left(  T_{2}\right)
\]
defines an element $\varphi_{\lambda}\in \mathcal{F}\left(  \mathcal{A}%
_{1}\otimes_{\text{\textrm{min}}}\mathcal{A}_{2},\mathcal{B}_{1}%
\otimes_{\text{\textrm{min}}}\mathcal{B}_{2}\right)  $. Moreover,%
\[
\pi \left(  \varphi_{\lambda}\left(  T_{1}\otimes T_{2}\right)  \right)
=\pi_{1}\left(  \varphi_{\lambda,1}\left(  T_{1}\right)  \right)  \otimes
\pi_{2}\left(  \varphi_{\lambda,2}\left(  T_{2}\right)  \right)
\rightarrow \pi \left(  E\left(  T_{1}\otimes T_{2}\right)  \right)
\]
in the weak operator topology on $B\left(  H_{1}\otimes H_{2}\right)  .$ Hence
$\mathcal{A}_{1}\otimes_{\text{\textrm{min}}}\mathcal{A}_{2}$ is strongly
injective in $\mathcal{B}_{1}\otimes_{\text{\textrm{min}}}\mathcal{B}_{2}$.

$\left(  7\right)  $ and $\left(  8\right)  $ follow from $\left(  6\right)  $.
\end{proof}

\bigskip

Suppose $\mathcal{A}$ is a unital C*-subalgebra of a unital C*-algebra
$\mathcal{B}$. We define two seminorms on $\mathcal{B}$ as follows:%
\[
d_{n}\left(  T,\mathcal{A},\mathcal{B}\right)  =\sup \left \{  \left \Vert
WT-TW\right \Vert :W\in \left(  \mathcal{A},\mathcal{B}\right)  ^{\prime
},\left \Vert W\right \Vert \leq1\right \}  ,
\]
and%
\[
d_{an}\left(  T,\mathcal{A},\mathcal{B}\right)  =\sup_{\left \{  W_{\lambda
}\right \}  }\limsup_{\lambda}\left \Vert W_{\lambda}T-TW_{\lambda}\right \Vert
\]
taken over all nets $\left \{  W_{\lambda}\right \}  $ of contractions in
$\mathcal{B}$ for which $\left \Vert AW_{\lambda}-W_{\lambda}A\right \Vert
\rightarrow0$ for every $A\in \mathcal{A}$.

The following lemma is obvious and the proof is omitted.

\begin{lemma}
Suppose $\mathcal{A}$ is a unital norm closed subalgebra of a unital
C*-algebra $\mathcal{B}$ and $T\in \mathcal{B}$. Then

\begin{enumerate}
\item $d_{n}\left(  T,\mathcal{A},\mathcal{B}\right)  $ and $d_{an}\left(
T,\mathcal{A},\mathcal{B}\right)  $ are seminorms on $\mathcal{B}$,

\item $d_{n}\left(  T^{\ast},\mathcal{A},\mathcal{B}\right)  =d_{n}\left(
T,\mathcal{A},\mathcal{B}\right)  $ and $d_{an}\left(  T^{\ast},\mathcal{A}%
,\mathcal{B}\right)  =d_{an}\left(  T,\mathcal{A},\mathcal{B}\right)  ,$

\item $d_{n}\left(  T,\mathcal{A},\mathcal{B}\right)  \leq d_{an}\left(
T,\mathcal{A},\mathcal{B}\right)  \leq2dist\left(  T,\mathcal{A}\right)
\leq2\left \Vert T\right \Vert $

\item $d_{n}\left(  T,\mathcal{A},\mathcal{B}\right)  =0$ if and only if
$T\in \left(  \mathcal{A},\mathcal{B}\right)  ^{\prime \prime}$

\item $d_{an}\left(  T,\mathcal{A},\mathcal{B}\right)  =0$ if and only if
$T\in Appr\left(  \mathcal{A},\mathcal{B}\right)  ^{\prime \prime}$
\end{enumerate}
\end{lemma}

\bigskip

We define $K_{n}\left(  \mathcal{A},\mathcal{B}\right)  $ and $K_{an}\left(
\mathcal{A},\mathcal{B}\right)  $ by%
\[
K_{n}\left(  \mathcal{A},\mathcal{B}\right)  =\sup \left \{  dist\left(
T,\mathcal{A}\right)  :T\in \mathcal{B},d_{n}\left(  T,\mathcal{A}%
,\mathcal{B}\right)  \leq1\right \}  ,
\]%
\[
K_{an}\left(  \mathcal{A},\mathcal{B}\right)  =\sup \left \{  dist\left(
T,\mathcal{A}\right)  :T\in \mathcal{B},d_{an}\left(  T,\mathcal{A}%
,\mathcal{B}\right)  \leq1\right \}  .
\]
Clearly $K_{n}\left(  \mathcal{A},\mathcal{B}\right)  $ is the smallest
$M\geq0$ such that, for every $T\in \mathcal{B}$, we have $dist\left(
T,\mathcal{A}\right)  \leq Md_{n}\left(  T,\mathcal{A},\mathcal{B}\right)  $
and $K_{an}\left(  \mathcal{A},\mathcal{B}\right)  $ is the smallest $N\geq0$
such that, for every $T\in \mathcal{B}$, we have $dist\left(  T,\mathcal{A}%
\right)  \leq Nd_{an}\left(  T,\mathcal{A},\mathcal{B}\right)  $. We say that
$\mathcal{A}$ is \emph{metric normal} in $\mathcal{A}$ if $K_{n}\left(
\mathcal{A},\mathcal{B}\right)  <\infty$ and $\mathcal{A}$ is \emph{metric
approximately normal} if $K_{an}\left(  \mathcal{A},\mathcal{B}\right)
<\infty$. It is also clear that $K_{an}\left(  \mathcal{A},\mathcal{B}\right)
\leq K_{n}\left(  \mathcal{A},\mathcal{B}\right)  $, so metric normality
implies metric approximate normality.

The following proposition shows the relationship between strong injectivity
and metric normality.

\begin{proposition}
\label{dl}Suppose $\mathcal{A}_{1}\subseteq \mathcal{A}_{2}\subseteq
\cdots \subseteq \mathcal{A}_{m}$ is are unital inclusions of C*-algebras and
$\mathcal{A}_{k}$ is weakly injective in $\mathcal{A}_{k+1}$ for $1\leq k<m$.
Then%
\[
K_{n}\left(  \mathcal{A}_{1},\mathcal{A}_{m}\right)  \leq1.
\]

\end{proposition}

\begin{proof}
For each $k$, $2\leq k\leq m$ choose a net $\left \{  \varphi_{\lambda
,k}\right \}  $ in $\mathcal{F}\left(  \mathcal{A}_{k-1},\mathcal{A}%
_{k}\right)  $, a conditional expectation $E_{k}:\mathcal{A}_{k}%
\rightarrow \mathcal{A}_{k-1}$ and a faithful representation $\pi
_{k}:\mathcal{A}_{k}\rightarrow B\left(  H_{k}\right)  $ such that%
\[
\pi_{k}\left(  \varphi_{\lambda,k}\left(  T\right)  \right)  \rightarrow
\pi_{k}\left(  E_{k}\left(  T\right)  \right)
\]
in the weak operator topology for every $T\in \mathcal{A}_{k}$. It is clear
that
\[
\mathcal{F}\left(  \mathcal{A}_{k-1},\mathcal{A}_{k}\right)  \subseteq
\mathcal{F}\left(  \mathcal{A}_{1},\mathcal{A}_{m}\right)
\]
for $2\leq k\leq m$. Morover, if $U$ is unitary and $U\in \left(
\mathcal{A}_{1},\mathcal{A}_{m}\right)  ^{\prime}$, then%
\[
\left \Vert T-UTU^{\ast}\right \Vert =\left \Vert TU-UT\right \Vert \leq
d_{n}\left(  T,\mathcal{A}_{1},\mathcal{A}_{m}\right)
\]
for all $T\in \mathcal{A}_{m}$. Suppose $T\in \mathcal{A}_{m}$, and let $B$
denote the closed ball in $\mathcal{A}_{m}$ centered at $T$ with radius
$d_{n}\left(  T,\mathcal{A}_{1},\mathcal{A}_{m}\right)  $. Let $\mathcal{W}%
_{m}$ denote the set of all $A\in \mathcal{A}_{m}$ such that $\pi_{m}\left(
A\right)  $ is in the weak-operator closure of the convex hull of $\left \{
\pi_{m}\left(  UTU^{\ast}\right)  :U\in \left(  \mathcal{A}_{1},\mathcal{A}%
_{m}\right)  ^{\prime},\text{ }U~\text{is unitary}\right \}  $. Clearly
$\mathcal{W}_{m}$ is convex and closed under conjugation by unitaries in
$\left(  \mathcal{A}_{1},\mathcal{A}_{m}\right)  ^{\prime}$, and, since
$\pi_{m}$ is an isometry, $\mathcal{W}_{m}\subseteq B$. It follows that
$E_{m}\left(  T\right)  \in \mathcal{W}_{m}$. Next we let $\mathcal{W}_{m-1}$
denote the set of all $A\in \mathcal{A}_{m}$ such that $\pi_{m-1}\left(
A\right)  $ is in the weak-operator closure of the convex hull of $\left \{
\pi_{m-1}\left(  UE_{m}\left(  T\right)  U^{\ast}\right)  :U\in \left(
\mathcal{A}_{1},\mathcal{A}_{m}\right)  ^{\prime},\text{ }U~\text{is
unitary}\right \}  $. Clearly $\mathcal{W}_{m-1}$ is convex and closed under
conjugation by unitaries in $\left(  \mathcal{A}_{1},\mathcal{A}_{m}\right)
^{\prime}$, and, since $\pi_{m-1}$ is an isometry, $\mathcal{W}_{m-1}\subseteq
B$, and it follows that $E_{m-1}\left(  E_{m}\left(  T\right)  \right)
\in \mathcal{W}_{m-1}\subseteq B$. Proceeding inductively, we see that
\[
E_{2}\left(  E_{3}\left(  \cdots E_{m}\left(  T\right)  \right)  \right)  \in
B\cap \mathcal{A}_{1},
\]
from which it follows that
\[
dist\left(  T,\mathcal{A}_{1}\right)  \leq d_{n}\left(  T,\mathcal{A}%
_{1},\mathcal{A}_{m}\right)  .
\]
Hence $K_{n}\left(  \mathcal{A}_{1},\mathcal{A}_{m}\right)  \leq1.$
\end{proof}

\bigskip

The following corollaries follow from Proposition $1$ and Proposition $3$.

\begin{corollary}
If $\mathcal{B}$ is a von Neumann algebra and $\mathcal{A}$ is a normal von
Neumann subalgebra such that $\left(  \mathcal{A},\mathcal{B}\right)
^{\prime}$ is hyperfinite, then $\mathcal{A}$ is metric normal in
$\mathcal{B}$ and $K_{n}\left(  \mathcal{M}_{k}\left(  \mathcal{A}\right)
,\mathcal{M}_{k}\left(  \mathcal{B}\right)  \right)  \leq1$ for every
$k\in \mathbb{N}$.
\end{corollary}

\bigskip

\begin{corollary}
If $\mathcal{A}$ is a maximal abelian selfadjoint subalgebra of a von Neumann
algebra $\mathcal{B}$, then $K_{n}\left(  \mathcal{M}_{k}\left(
\mathcal{A}\right)  ,\mathcal{M}_{k}\left(  \mathcal{B}\right)  \right)
\leq1$ for every $k\in \mathbb{N}$.
\end{corollary}

\bigskip

\begin{corollary}
If $\mathcal{A}$ is a maximal abelian selfadjoint subalgebra of a von Neumann
algebra $\mathcal{B}$, and $\mathcal{W}$ is any von Neumann algebra, then
$K_{n}\left(  \mathcal{W}\otimes \mathcal{A},\mathcal{W}\otimes \mathcal{B}%
\right)  \leq1,$ where $\otimes$ denotes the spatial tensor product.\bigskip
\end{corollary}

\begin{corollary}
If $\mathcal{B}$ is a hyperfinite von Neumann algebra, then every normal von
Neumann subalgebra $\mathcal{A}$ of $\mathcal{B}$ is metric normal and%
\[
K_{n}\left(  \mathcal{A},\mathcal{B}\right)  \leq1.
\]

\end{corollary}

Without injectivity, this is the best analogue of Proposition $\ref{dl}$.

\begin{lemma}
\label{rel}If $\mathcal{A}\subseteq \mathcal{D}\subseteq \mathcal{B}$ are unital
C*-algebras, then%

\[
K_{n}\left(  \mathcal{A},\mathcal{B}\right)  \leq K_{n}\left(  \mathcal{D}%
,\mathcal{B}\right)  +K_{n}\left(  \mathcal{A},\mathcal{D}\right)  \left(
2K_{n}\left(  \mathcal{D},\mathcal{B}\right)  +1\right)  ,
\]
and%
\[
K_{an}\left(  \mathcal{A},\mathcal{B}\right)  \leq K_{an}\left(
\mathcal{D},\mathcal{B}\right)  +K_{an}\left(  \mathcal{A},\mathcal{D}\right)
\left(  2K_{an}\left(  \mathcal{D},\mathcal{B}\right)  +1\right)  .
\]

\end{lemma}

\begin{proof}
We present the proof for $K_{n}$; the proof for $K_{an}$ is similar. Suppose
$T\in \mathcal{B}$ and $\varepsilon>0.$ Then
\[
dist\left(  T,\mathcal{D}\right)  <\left[  K_{n}\left(  \mathcal{D}%
,\mathcal{B}\right)  +\varepsilon \right]  d_{n}\left(  T,\mathcal{D}%
,\mathcal{B}\right)  .
\]
Hence there is a $D\in \mathcal{D}$ such that%
\[
\left \Vert T-D\right \Vert \leq \left[  K_{n}\left(  \mathcal{D},\mathcal{B}%
\right)  +\varepsilon \right]  d_{n}\left(  T,\mathcal{D},\mathcal{B}\right)
.
\]
Similarly, there is an $A\in \mathcal{A}$ such that%
\[
\left \Vert D-A\right \Vert \leq \left[  K_{n}\left(  \mathcal{A},\mathcal{D}%
\right)  +\varepsilon \right]  d_{n}\left(  D,\mathcal{A},\mathcal{D}\right)
.
\]
Hence%
\[
\left \Vert T-A\right \Vert \leq \left \Vert T-D\right \Vert +\left \Vert
D-A\right \Vert \leq
\]%
\[
\left[  K_{n}\left(  \mathcal{D},\mathcal{B}\right)  +\varepsilon \right]
d_{n}\left(  T,\mathcal{D},\mathcal{B}\right)  +\left[  K_{n}\left(
\mathcal{A},\mathcal{D}\right)  +\varepsilon \right]  d_{n}\left(
D,\mathcal{A},\mathcal{D}\right)  .
\]
However,
\[
d_{n}\left(  T,\mathcal{D},\mathcal{B}\right)  \leq d_{n}\left(
T,\mathcal{A},\mathcal{B}\right)  ,
\]
and%
\[
d_{n}\left(  D,\mathcal{A},\mathcal{D}\right)  \leq d_{n}\left(
D,\mathcal{A},\mathcal{B}\right)  \leq2\left \Vert T-D\right \Vert +d_{n}\left(
T,\mathcal{A},\mathcal{B}\right)  \leq
\]%
\[
2\left[  K_{n}\left(  \mathcal{D},\mathcal{B}\right)  +\varepsilon \right]
d_{n}\left(  T,\mathcal{D},\mathcal{B}\right)  +d_{n}\left(  T,\mathcal{A}%
,\mathcal{B}\right)  \leq
\]%
\[
\left(  2\left[  K_{n}\left(  \mathcal{D},\mathcal{B}\right)  +\varepsilon
\right]  +1\right)  d_{n}\left(  T,\mathcal{A},\mathcal{B}\right)
\]

Hence%
\[
dist\left(  T,\mathcal{A}\right)  \leq \left \Vert T-A\right \Vert \leq
\]%
\[
\left[  K_{n}\left(  \mathcal{D},\mathcal{B}\right)  +\varepsilon \right]
d_{n}\left(  T,\mathcal{A},\mathcal{B}\right)  +\left[  K_{n}\left(
\mathcal{A},\mathcal{D}\right)  +\varepsilon \right]  \left(  2\left[
K_{n}\left(  \mathcal{D},\mathcal{B}\right)  +\varepsilon \right]  +1\right)
d_{n}\left(  T,\mathcal{A},\mathcal{B}\right)  .
\]
Letting $\varepsilon \rightarrow0^{+}$, we see%
\[
dist\left(  T,\mathcal{A}\right)  \leq
\]%
\[
\left[  K_{n}\left(  \mathcal{D},\mathcal{B}\right)  +K_{n}\left(
\mathcal{A},\mathcal{D}\right)  \left(  2K_{n}\left(  \mathcal{D}%
,\mathcal{B}\right)  +1\right)  \right]  d_{n}\left(  T,\mathcal{A}%
,\mathcal{B}\right)  .
\]
It follows that%
\[
K_{n}\left(  \mathcal{A},\mathcal{B}\right)  \leq K_{n}\left(  \mathcal{D}%
,\mathcal{B}\right)  +K_{n}\left(  \mathcal{A},\mathcal{D}\right)  \left(
2K_{n}\left(  \mathcal{D},\mathcal{B}\right)  +1\right)
\]

\end{proof}

\bigskip

\bigskip

We now consider the metric approximate normality for direct limits.

\begin{lemma}
\label{andl}Suppose $\mathcal{A}$ is a unital C*-subalgebra of a unital
C*-algebra $\mathcal{B}$ and $\left \{  \mathcal{A}_{i}:i\in I\right \}  $ is an
increasingly directed family of C*-subalgebras of $\mathcal{A}$. If
$\mathcal{A}$ is the norm closure of $\cup_{i\in I}\mathcal{A}_{i}$, then%
\[
K_{an}\left(  \mathcal{A},\mathcal{B}\right)  \leq \liminf_{i}K_{an}\left(
\mathcal{A}_{i},\mathcal{B}\right)  \text{.}%
\]

\end{lemma}

\begin{proof}
Suppose $T\in \mathcal{B}$, $F\subseteq \mathcal{A}$ is finite, $\varepsilon>0$,
and let $\lambda=\left(  F,\varepsilon \right)  $. Then
\[
dist\left(  T,\mathcal{A}\right)  =\lim_{i}dist\left(  T,\mathcal{A}%
_{i}\right)  \leq
\]%
\[
\sup_{i}K_{an}\left(  T,\mathcal{A}_{i}\right)  \liminf_{i}d_{n}\left(
T,\mathcal{A}_{i},\mathcal{B}\right)  .
\]
We can choose $i_{0}$ sufficiently large so that there is a map $\alpha
:F\rightarrow \mathcal{A}_{i_{0}}$ such that
\[
\left \Vert A-\alpha \left(  A\right)  \right \Vert <\varepsilon/2
\]
for every $A\in F$ and so that
\[
\liminf_{i}d_{an}\left(  T,\mathcal{A}_{i},\mathcal{B}\right)  \leq
d_{an}\left(  T,\mathcal{A}_{i_{0}},\mathcal{B}\right)  +\varepsilon.
\]
We can choose a unitary $U_{\lambda}$ in $\mathcal{B}$ so that%
\[
\left \Vert U_{\lambda}\alpha \left(  A\right)  -\alpha \left(  A\right)
U_{\lambda}\right \Vert <\varepsilon/3
\]
for every $A\in F$ and so that%
\[
d_{an}\left(  T,\mathcal{A}_{i_{0}},\mathcal{B}\right)  \leq \left \Vert
U_{\lambda}T-TU_{\lambda}\right \Vert +\varepsilon.
\]
It follows that%
\[
dist\left(  T,\mathcal{A}\right)  \leq \left[  \left \Vert U_{\lambda
}T-TU_{\lambda}\right \Vert +2\varepsilon \right]  \sup_{i}K_{an}\left(
\mathcal{A}_{i},\mathcal{B}\right)  .
\]
and
\[
\left \Vert U_{\lambda}A-AU_{\lambda}\right \Vert \leq \varepsilon.
\]
If we let $\Lambda$ be the set of all pairs $\lambda=\left(  F,\varepsilon
\right)  $ directed by $\left(  \subseteq,\geq \right)  $, we see that
$\left \{  U_{\lambda}\right \}  $ is a net such that%
\[
\left \Vert AU_{\lambda}-U_{\lambda}A\right \Vert \rightarrow0
\]
for every $A\in \mathcal{A}$ and such that%
\[
dist\left(  T,\mathcal{A}\right)  \leq \left[  \lim_{\lambda}\left \Vert
U_{\lambda}T-TU_{\lambda}\right \Vert \right]  \sup_{i}K_{an}\left(
\mathcal{A}_{i},\mathcal{B}\right)  \leq d_{an}\left(  T,\mathcal{A}%
,\mathcal{B}\right)  \sup_{i}K_{an}\left(  \mathcal{A}_{i},\mathcal{B}\right)
.
\]
Hence $K_{an}\left(  \mathcal{A},\mathcal{B}\right)  \leq \sup_{i}K_{an}\left(
\mathcal{A}_{i},\mathcal{B}\right)  $, and since the same holds for when we
restrict to the set $\left \{  \mathcal{A}_{i}:i\geq j\right \}  $ for some $j,$
we can replace $\sup_{i}K_{an}\left(  \mathcal{A}_{i},\mathcal{B}\right)  $
with $\liminf_{i}K_{an}\left(  \mathcal{A}_{i},\mathcal{B}\right)  $.
\end{proof}

\bigskip

We now want to extend a key result in \cite{H4}. Recall from \cite{H4} that a
unital C*-algebra $\mathcal{B}$ is \emph{centrally prime} if, whenever $0\leq
x,y\leq1$ are in $\mathcal{B}$ and $x\mathcal{B}y=\left \{  0\right \}  $, then
there is an $e\in \mathcal{Z}\left(  \mathcal{B}\right)  $ such that $x\leq
e\leq1$ and $y\leq1-e$. \ 

The following result is a generalization of \cite[Theorem 1]{H4}, in which
$\mathcal{W}=\mathbb{C}$. That result required S. Machado's metric version
\cite{M} of the Bishop-Stone-Weierstrass theorem \cite{B}. Here we require
Machado's vector version of his result \cite{M} (See \cite{R} for an beautiful
elementary proof.)\bigskip

\begin{proposition}
\label{Macapp}Suppose $\mathcal{A}\subseteq \mathcal{D}$ are unital commutative
C*-algebras, $\mathcal{W}$ is a unital C*-algebra, $\mathcal{B}$ is a
centrally prime unital C*-algebra such that

\begin{enumerate}
\item $\mathcal{A}\otimes \mathcal{W}\subseteq \mathcal{D}\otimes \mathcal{W}%
\subseteq \mathcal{B\otimes}_{\text{\textrm{min}}}\mathcal{W},$

\item $\mathcal{Z}\left(  \mathcal{B\otimes}_{\text{\textrm{min}}}%
\mathcal{W}\right)  \subseteq \mathcal{A\otimes W}$.
\end{enumerate}

Then, for every $T\in \mathcal{D}\otimes \mathcal{W}$,%
\[
dist\left(  T,\mathcal{A\otimes W}\right)  \leq d_{an}\left(
T,\mathcal{A\otimes W},\mathcal{B}\otimes_{\text{\textrm{min}}}\mathcal{W}%
\right)  .
\]

\end{proposition}

\begin{proof}
We can view $\mathcal{D}=C\left(  X\right)  $ for some compact Hausdorff space
$X$ and we can view $\mathcal{D\otimes W}$ as $C\left(  X,\mathcal{W}\right)
,$ the C*-algebra of continuous functions from $X$ to $\mathcal{W}$. We can
write $T=f\in C\left(  X,\mathcal{W}\right)  $. It follows that $\mathcal{A}%
\otimes \mathcal{W}$ is a C*-subalgebra of $C\left(  X,\mathcal{W}\right)  $
that is an $\mathcal{A}$-module. It follows from Machado's theorem \cite{M},
that there is a closed $\mathcal{A}$-antisymmetric subset $E\subseteq X$ such
that%
\[
dist\left(  f,\mathcal{A}\otimes \mathcal{W}\right)  =dist\left(
f|_{E},\left(  \mathcal{A}\otimes \mathcal{W}\right)  |_{E}\right)  .
\]
However, since $\mathcal{A=A}^{\ast}$ and $E$ is $\mathcal{A}$-antisymmetric,
we see that every function in $\mathcal{A}$ is constant on $E$. Hence, if
$u\in \mathcal{A}$ and $w\in \mathcal{W}$ we have, for every $x\in X$ that
$\left(  u\otimes w\right)  \left(  x\right)  =u\left(  x\right)  w$. Hence,
every function in $\mathcal{A}\otimes \mathcal{W}$ is constant on $E$. Thus%
\[
dist\left(  f|_{E},\left(  \mathcal{A}\otimes \mathcal{W}\right)  |_{E}\right)
=\inf \left \{  \left \Vert f|_{E}-h\right \Vert :h:E\rightarrow \mathcal{W},\text{
}h\text{ is constant}\right \}  =
\]%
\[
\inf_{w\in \mathcal{W}}\sup_{x\in E}\left \Vert f\left(  x\right)  -w\right \Vert
.
\]
Since $E$ is compact, we can choose $\alpha,\beta \in E$ such that%
\[
\left \Vert f\left(  \alpha \right)  -f\left(  \beta \right)  \right \Vert
=\sup_{x,y\in E}\left \Vert f\left(  x\right)  -f\left(  y\right)  \right \Vert
.
\]
If we let $w=f\left(  \beta \right)  ,$ we see that%
\[
dist\left(  f|_{E},\left(  \mathcal{A}\otimes \mathcal{W}\right)  |_{E}\right)
\leq \sup_{x\in E}\left \Vert f\left(  x\right)  -f\left(  \beta \right)
\right \Vert =\left \Vert f\left(  \alpha \right)  -f\left(  \beta \right)
\right \Vert \text{.}%
\]
Hence,%
\[
dist\left(  f,\mathcal{A}\otimes \mathcal{W}\right)  \leq \left \Vert f\left(
\beta \right)  -f\left(  \alpha \right)  \right \Vert .
\]

If $f\left(  \alpha \right)  =f\left(  \beta \right)  ,$ then $T=f\in
\mathcal{A}\otimes \mathcal{W}$ and the desired inequality holds. Hence we can
assume $\alpha \neq \beta$. Let $\Lambda$ be the set of pairs $\left(
U,V\right)  $, we $U$ and $V$ are disjoint open subsets of $X$ such that
$\alpha \in U$ and $\beta \in V$. Suppose $\lambda=\left(  U,V\right)
\in \Lambda$. We can define $g_{\lambda},h_{\lambda},r_{\lambda},s_{\lambda}\in
C\left(  X\right)  $ such that

\begin{enumerate}
\item $0\leq g_{\lambda},h_{\lambda},r_{\lambda},s_{\lambda}\leq1$

\item $g_{\lambda}\left(  \alpha \right)  =h_{\lambda}\left(  \alpha \right)
=1,$ $g_{\lambda}h_{\lambda}=h_{\lambda}$, $g_{\lambda}|_{X\backslash
U_{\lambda}}=0$

\item $r_{\lambda}\left(  \beta \right)  =s_{\lambda}\left(  \beta \right)  =1,$
$r_{\lambda}s_{\lambda}=s_{\lambda},$ $r_{\lambda}|_{X\backslash V}=0.$
\end{enumerate}

We then have, for every $F\in C\left(  X,\mathcal{W}\right)  $

\begin{enumerate}
\item[4.] $h_{\lambda}F=Fh_{\lambda}$ and $\left \Vert h_{\lambda}F-\left(
1\otimes F\left(  \alpha \right)  \right)  h_{\lambda}\right \Vert \rightarrow0$

\item[5.] $s_{\lambda}F=Fs_{\lambda}$ and $\left \Vert s_{\lambda}F-\left(
1\otimes F\left(  \beta \right)  \right)  s_{\lambda}\right \Vert \rightarrow0.$
\end{enumerate}

Claim: $h_{\lambda}\left(  \mathcal{B\otimes}1\right)  s_{\lambda}=\left(
h_{\lambda}\mathcal{B}s_{\lambda}\right)  \otimes1\neq \left \{  0\right \}  .$
Since $\mathcal{B}$ is centrally prime, $h_{\lambda}\mathcal{B}s_{\lambda
}=\left \{  0\right \}  $ implies that there is an $e\in \mathcal{Z}\left(
\mathcal{B}\right)  \subseteq \mathcal{A}$ such that $h_{\lambda}\leq e\leq1$
and $s_{\lambda}\leq1-e\leq1$. Thus $1\leq e\left(  \alpha \right)  $ and
$0\leq e\left(  \beta \right)  ,$ which contradicts $e\left(  \alpha \right)
=e\left(  \beta \right)  $. This proves the claim.

For each $\lambda \in \Lambda$ we can choose $Q_{\lambda}\in h_{\lambda
}\mathcal{B}s_{\lambda}\otimes1$ with $\left \Vert Q_{\lambda}\right \Vert =1$.
We then have, for every $F\in C\left(  X,\mathcal{W}\right)  $%
\[
\left \Vert \left[  FQ_{\lambda}-Q_{\lambda}F\right]  -\left[  \left(  1\otimes
F\left(  \alpha \right)  \right)  Q_{\lambda}-Q_{\lambda}\left(  1\otimes
F\left(  \beta \right)  \right)  \right]  \right \Vert \rightarrow0,
\]
so
\[
\left \vert \left \Vert FQ_{\lambda}-Q_{\lambda}F\right \Vert -\left \Vert \left(
1\otimes F\left(  \alpha \right)  \right)  Q_{\lambda}-Q_{\lambda}\left(
1\otimes F\left(  \beta \right)  \right)  \right \Vert \right \vert
\rightarrow0.
\]
However,%
\[
\left(  1\otimes F\left(  \alpha \right)  \right)  Q_{\lambda}=Q_{\lambda
}\otimes F\left(  \alpha \right)  ,\text{ and }Q_{\lambda}\left(  1\otimes
F\left(  \beta \right)  \right)  =Q_{\lambda}\otimes F\left(  \beta \right)  .
\]
Hence, $\left \Vert FQ_{\lambda}-Q_{\lambda}F\right \Vert \rightarrow0$ for
every $F\in \mathcal{A\otimes W}$ and
\[
\lim_{\lambda}\left \Vert fQ_{\lambda}-Q_{\lambda}f\right \Vert =\lim_{\lambda
}\left \Vert Q_{\lambda}\otimes \left(  f\left(  \beta \right)  -f\left(
\alpha \right)  \right)  \right \Vert =
\]%
\[
\left \Vert f\left(  \beta \right)  -f\left(  \alpha \right)  \right \Vert \geq
dist\left(  f,\mathcal{A}\otimes \mathcal{W}\right)  .
\]

\end{proof}

\bigskip

\begin{corollary}
Suppose $\mathcal{A}$ is a commutative unital C*-subalgebra of a centrally
prime unital C*-algebra $\mathcal{B}$ such that $\mathcal{Z}\left(
\mathcal{B}\right)  \subseteq \mathcal{A}$ and $\mathcal{W}$ is any unital
C*-algebra. Then $\mathcal{A}\otimes \mathcal{W}$ is approximately normal in
$\mathcal{B}\otimes_{\text{\textrm{min}}}\mathcal{W}$.
\end{corollary}

\begin{proof}
Suppose $\mathcal{D}\subseteq \mathcal{B}$ is a masa in $\mathcal{B}$ that
contains $\mathcal{A}$. It follows that $\left(  \mathcal{D}\otimes
\mathcal{W},\mathcal{B}\otimes_{\text{\textrm{min}}}\mathcal{W}\right)
^{\prime}=\mathcal{D}\otimes \mathcal{Z}\left(  \mathcal{W}\right)  ,$ and
$\left(  \mathcal{D}\otimes \mathcal{W},\mathcal{B}\otimes_{\text{\textrm{min}%
}}\mathcal{W}\right)  ^{\prime \prime}=\mathcal{D}\otimes \mathcal{W}$. Hence
$\mathcal{D}\otimes \mathcal{W}$ is normal in $\mathcal{B}\otimes
_{\text{\textrm{min}}}\mathcal{W}$. Hence, if $T\in Appr\left(  \mathcal{A}%
\otimes \mathcal{W},\mathcal{B}\otimes_{\text{\textrm{min}}}\mathcal{W}\right)
^{\prime \prime}$, then $T\in \mathcal{D}\otimes \mathcal{W}$, and it follows
from Proposition \ref{Macapp} that $T\in \mathcal{A}\otimes \mathcal{W}$.
\end{proof}

\bigskip

\begin{corollary}
Suppose $\mathcal{A}$ is a commutative unital C*-subalgebra of a von Neumann
algebra $\mathcal{B}$ and $\mathcal{Z}\left(  \mathcal{B}\right)
\subseteq \mathcal{A}$, and $\mathcal{W}$ is any unital C*-algebra. Then
$\mathcal{A\otimes_{\text{\textrm{min}}}W}$ is metric approximately normal in
$\mathcal{B\otimes_{\text{\textrm{min}}}W}$ and
\[
K_{an}\left(  \mathcal{A\otimes_{\text{\textrm{min}}}W},\mathcal{B\otimes
_{\text{\textrm{min}}}W}\right)  \leq4.
\]

\end{corollary}

\begin{proof}
Let $\mathcal{D}$ be a masa in $\mathcal{B}$ that contains $\mathcal{A}$. It
follows from Proposition \ref{inj} that $\mathcal{D}$ is weak*-injective in
$\mathcal{B}$, and that $\mathcal{D}\otimes \mathcal{W}$ is strongly injective
in $\mathcal{B\otimes_{\text{\textrm{min}}}W}$. Suppose $T\in \mathcal{B\otimes
_{\text{\textrm{min}}}W}$. We can assume that $\mathcal{B\subseteq B}\left(
H\right)  $ is a von Neumann algebra and $\mathcal{W}\subseteq B\left(
M\right)  $ and $\mathcal{B\otimes_{\text{\textrm{min}}}W}$ is the spatial
tensor product of $\mathcal{B}$ and $\mathcal{W}$ in $B\left(  H\otimes
M\right)  $. Then there is a net $\left \{  \varphi_{\lambda}\right \}  $ in
$\mathcal{F}\left(  \mathcal{D},\mathcal{B}\right)  $ such that $E\left(
S\right)  =w^{\ast}$-$\lim{}_{\lambda}\varphi_{\lambda}\left(  S\right)  $ is
a conditional expectation from $\mathcal{B}$ to $\mathcal{D}$. Then
$E\otimes1:\mathcal{B\otimes_{\text{\textrm{min}}}W\rightarrow D}%
\otimes \mathcal{W}$ defined, for every $R$ in $B\otimes W$, by
\[
\left(  E\otimes1\right)  \left(  R\right)  =w^{\ast}\text{-}\lim_{\lambda
}\left(  \varphi_{\lambda}\otimes1\right)  \left(  R\right)
\]
is a conditional expectation and each
\[
\varphi_{\lambda}\otimes1\in \mathcal{F}\left(  \mathcal{D}\otimes
\mathcal{W},\mathcal{B\otimes_{\text{\textrm{min}}}W}\right)  \subseteq
\mathcal{F}\left(  \mathcal{A\otimes W},\mathcal{B\otimes_{\text{\textrm{min}%
}}W}\right)  .
\]
Hence $T_{1}=\left(  E\otimes1\right)  \left(  T\right)  \in B$, where $B$ is
the closed ball in $\mathcal{B\otimes_{\text{\textrm{min}}}W}$ centered at $T$
with radius $d_{n}\left(  T,\mathcal{A}\otimes \mathcal{W},\mathcal{B\otimes
_{\text{\textrm{min}}}W}\right)  $. However, Theorem $\ref{Macapp}$ implies
that%
\[
dist\left(  T_{1},\mathcal{A\otimes W}\right)  \leq d_{an}\left(
T_{1},\mathcal{A\otimes W},\mathcal{B}\otimes_{\text{\textrm{min}}}%
\mathcal{W}\right)  \leq
\]%
\[
\leq d_{an}\left(  T,\mathcal{A\otimes W},\mathcal{B}\otimes
_{\text{\textrm{min}}}\mathcal{W}\right)  +2\left \Vert T-T_{1}\right \Vert \leq
\]%
\[
d_{an}\left(  T,\mathcal{A}\otimes \mathcal{W},\mathcal{B\otimes
_{\text{\textrm{min}}}W}\right)  +2d_{n}\left(  T,\mathcal{A}\otimes
\mathcal{W},\mathcal{B\otimes_{\text{\textrm{min}}}W}\right)  \leq
\]%
\[
3d_{an}\left(  T,\mathcal{A}\otimes \mathcal{W},\mathcal{B\otimes
_{\text{\textrm{min}}}W}\right)  .
\]
Hence
\[
dist\left(  T,\mathcal{A}\otimes \mathcal{W}\right)  \leq dist\left(
T_{1},\mathcal{A}\otimes \mathcal{W}\right)  +\left \Vert T-T_{1}\right \Vert
\leq
\]%
\[
4d_{an}\left(  T,\mathcal{A}\otimes \mathcal{W},\mathcal{B\otimes
_{\text{\textrm{min}}}W}\right)  .
\]

\end{proof}

\bigskip

\bigskip \ 

\begin{theorem}
\label{smallAH}If $\mathcal{B}$ is a unital centrally prime C*-algebra and
$\mathcal{Z}\left(  \mathcal{B}\right)  \subseteq \mathcal{A}$ is a unital
C*-subalgebra that is isomorphic to a finite direct sum of tensor products of
algebras of the form $\mathcal{D}\otimes \mathcal{M}_{k}\left(  \mathbb{C}%
\right)  $, with $\mathcal{D}$ commutative, then $\mathcal{A}$ is
approximately normal in $\mathcal{B}$. Moreover, if $\mathcal{B}$ is a von
Neumann algebra, then $K_{an}\left(  \mathcal{A},\mathcal{B}\right)  \leq4.$
\end{theorem}

\begin{proof}
Write $\mathcal{A}=\mathcal{A}_{1}\oplus \cdots \oplus \mathcal{A}_{n}$ where
each $\mathcal{A}_{k}$ is isomorphic to $\mathcal{D}_{k}\otimes \mathcal{M}%
_{s_{k}}\left(  \mathbb{C}\right)  $ for some $s_{k}$ in $\mathbb{N}$, and let
$P_{1}=1\oplus0\oplus \cdots \oplus0,P_{2}=0\oplus1\oplus \cdots \oplus
0,\ldots,P_{n}=0\oplus \cdots \oplus0\oplus1$. It follows from Proposition
\ref{inj} that $\sum_{j=1}^{n}P_{j}\mathcal{B}P_{j}$ is strongly injective in
$\mathcal{B}$. Since $\mathcal{D}_{k}\otimes \mathcal{M}_{s_{k}}\left(
\mathbb{C}\right)  \subseteq P_{k}\mathcal{B}P_{k}$ we can write
\[
P_{k}\mathcal{B}P_{k}=\mathcal{B}_{k}\otimes \mathcal{M}_{s_{k}}\left(
\mathbb{C}\right)
\]
with $\mathcal{D}_{k}\subseteq \mathcal{B}_{k}$. Since $\mathcal{B}$ is
centrally prime, so is each $P_{k}\mathcal{B}P_{k}$, and thus so does each
$\mathcal{B}_{k}$. Since $\mathcal{Z}\left(  \mathcal{B}\right)
\subseteq \mathcal{D}$, we know that
\[
\mathcal{Z}\left(  \mathcal{B}_{k}\otimes \mathcal{M}_{s_{k}}\left(
\mathbb{C}\right)  \right)  =\mathcal{Z}\left(  \mathcal{B}_{k}\right)
\otimes1\subseteq \mathcal{D}_{k}\otimes \mathcal{M}_{s_{k}}\left(
\mathbb{C}\right)  ,
\]
which implies $\mathcal{Z}\left(  \mathcal{B}_{k}\right)  \subseteq
\mathcal{D}_{k}$ for $1\leq k\leq n$. Since, by \cite{H4}, $\mathcal{D}_{k}$
is normal in $\mathcal{B}_{k}$, we know that $P_{k}\mathcal{A}P_{k}%
=\mathcal{D}_{k}\otimes \mathcal{M}_{s_{k}}\left(  \mathbb{C}\right)  $ is
normal in $\mathcal{B}_{k}\otimes \mathcal{M}_{s_{k}}\left(  \mathbb{C}\right)
=P_{k}\mathcal{B}P_{k}$ for $1\leq k\leq n$. Hence by Proposition \ref{inj},
$\mathcal{A}$ is normal in $\mathcal{B}$. If $\mathcal{B}$ is a von Neumann
algebra and if, for each $k$, $\mathcal{E}_{k}$ is a masa in $\mathcal{B}_{k}$
containing $\mathcal{D}_{k}$ for $1\leq k\leq n$, then%
\[
\sum_{1\leq k\leq n}^{\oplus}\mathcal{E}_{k}\otimes \mathcal{M}_{s_{k}}\left(
\mathbb{C}\right)
\]
is weak* injective in $\mathcal{B}$. It follows from Theorem \ref{Macapp}
that, for every $S=S_{1}\oplus \cdots \oplus S_{n}\in \sum_{1\leq k\leq
n}^{\oplus}\mathcal{E}_{k}\otimes \mathcal{M}_{s_{k}}\left(  \mathbb{C}\right)
$%
\[
dist\left(  S,\mathcal{A}\right)  \leq \max_{1\leq k\leq n}dist\left(
S_{k},\mathcal{A}_{k}\otimes \mathcal{M}_{s_{k}}\left(  \mathbb{C}\right)
\right)  \leq
\]%
\[
\max_{1\leq k\leq n}d_{an}\left(  S_{k},\mathcal{A}_{k}\mathcal{\otimes
M}_{s_{k}}\left(  \mathbb{C}\right)  ,\mathcal{B}_{k}\otimes \mathcal{M}%
_{s_{k}}\left(  \mathbb{C}\right)  \right)  \leq
\]%
\[
d_{an}\left(  S,\mathcal{A},\mathcal{B}\right)  .
\]
If $T\in \mathcal{B}$, it follows that there is an $S\in \sum_{1\leq k\leq
n}^{\oplus}\mathcal{E}_{k}\otimes \mathcal{M}_{s_{k}}\left(  \mathbb{C}\right)
$ such that%
\[
\left \Vert T-S\right \Vert \leq d_{an}\left(  T,\mathcal{A},\mathcal{B}\right)
.
\]
It follows that
\[
dist\left(  T,\mathcal{A}\right)  \leq dist\left(  S,\mathcal{A}\right)
+\left \Vert S-T\right \Vert \leq
\]%
\[
d_{an}\left(  S-T,\mathcal{A},\mathcal{B}\right)  +2d_{an}\left(
T,\mathcal{A},\mathcal{B}\right)  \leq
\]%
\[
2\left \Vert S-T\right \Vert +2d_{an}\left(  T,\mathcal{A},\mathcal{B}\right)
\leq4d_{an}\left(  T,\mathcal{A},\mathcal{B}\right)  .
\]

\end{proof}

\bigskip

\begin{theorem}
\label{AH}If $\mathcal{A}$ is a unital AH C*-subalgebra of a von Neumann
algebra $\mathcal{B}$, then%
\[
K_{an}\left(  \mathcal{A},\mathcal{B}\right)  \leq4.
\]

\end{theorem}

\begin{proof}
This follows from Theorem \ref{smallAH} and Lemma \ref{andl}.
\end{proof}

\bigskip

\bigskip

\begin{theorem}
\label{AF}If $\mathcal{B}$ is a primitive unital C*-algebra and $\mathcal{A}$
is a unital AF C*-subalgebra of $\mathcal{B}$, then%
\[
K_{an}\left(  \mathcal{A},\mathcal{B}\right)  \leq1.
\]
\bigskip
\end{theorem}

\begin{proof}
Suppose $\mathcal{A=M}_{s_{1}}\left(  \mathbb{C}\right)  \oplus \cdots
\oplus \mathcal{M}_{s_{k}}\left(  \mathbb{C}\right)  $ and let $P_{1}%
=1\oplus0\oplus \cdots \oplus0$, $P_{2}=0\oplus1\oplus \cdots \oplus0\,$,$\ldots
$,$P_{k}=0\oplus \cdots \oplus0\oplus1$. Then $\sum_{1\leq i\leq k}%
P_{i}\mathcal{B}P_{i}$ is strongly injective in $\mathcal{B}$ and we can write
$P_{i}\mathcal{B}P_{i}=\mathcal{M}_{s_{i}}\left(  \mathcal{B}_{i}\right)  $
for $1\leq i\leq k.$ Since $\mathcal{B}$ is primitive, it follows that each
$\mathcal{B}_{i}$ is primitive, and thus $\mathbb{C}$ is strongly injective in
$\mathcal{B}_{i}$ for $1\leq i\leq k$. Hence by Proposition \ref{inj},
$\mathcal{M}_{s_{i}}\left(  \mathbb{C}\right)  $ is strongly injective in
$\mathcal{M}_{s_{i}}\left(  \mathcal{B}_{i}\right)  $ for $1\leq i\leq k$.
Whence, by Proposition \ref{inj}, $\mathcal{A}$ is strongly injective in
$\sum_{1\leq i\leq k}P_{i}\mathcal{B}P_{i}$. Hence, by Proposition \ref{dl},
$K_{n}\left(  \mathcal{A},\mathcal{B}\right)  \leq1.$ The general case easily
follows from Lemma \ref{andl}.\bigskip
\end{proof}

The reason we can get better metric results (AH instead of AF) for von Neumann
algebras than primitive C*-algebras is that we know that every masa is
strongly injective, or that $K_{an}\left(  \mathcal{A},\mathcal{B}\right)
<\infty$ when $\mathcal{A}$ is a masa in a von Neumann algebra $\mathcal{B}$.

Suppose $I$ is an infinite set and $\left \{  \mathcal{B}_{i}:i\in I\right \}  $
is a family of unital C*-algebras and, for each $i\in I$, $\mathcal{A}_{i}$ is
a unital C*-subalgebra of $\mathcal{B}_{i}$. Suppose $\alpha$ is a nontrivial
ultrafilter on $I$ and $\pi:%
{\displaystyle \prod_{i\in I}}
\mathcal{B}_{i}\rightarrow%
{\displaystyle \prod_{i\in I}}
\mathcal{B}_{i}/\sum_{i\in I}^{\oplus}\mathcal{B}_{i}$ and $\rho:%
{\displaystyle \prod_{i\in I}}
\mathcal{B}_{i}\rightarrow \overset{\alpha}{%
{\displaystyle \prod}
}\mathcal{B}_{i}$ are the quotient maps, where $\overset{\alpha}{%
{\displaystyle \prod}
}\mathcal{B}_{i}$ is the C*-ultraproduct of the $\mathcal{B}_{i}$'s with
respect to the ultrafilter $\alpha$. Let $\mathcal{A}=%
{\displaystyle \prod_{i\in I}}
\mathcal{A}_{i}$.

\begin{proposition}
\label{ultra}The following are true.

\begin{enumerate}
\item $K_{an}\left(  \pi \left(  \mathcal{A}\right)  ,%
{\displaystyle \prod_{i\in I}}
\mathcal{B}_{i}/\sum_{i\in I}\mathcal{B}_{i}\right)  \leq \sup_{i\in I}%
K_{an}\left(  \mathcal{A}_{i},\mathcal{B}_{i}\right)  $

\item $K_{n}\left(  \pi \left(  \mathcal{A}\right)  ,%
{\displaystyle \prod_{i\in I}}
\mathcal{B}_{i}/\sum_{i\in I}\mathcal{B}_{i}\right)  \leq \sup_{i\in I}%
K_{n}\left(  \mathcal{A}_{i},\mathcal{B}_{i}\right)  $.

\item $K_{an}\left(  \rho \left(  \mathcal{A}\right)  ,\overset{\alpha}{%
{\displaystyle \prod}
}\mathcal{B}_{i}\right)  \leq \lim_{i\rightarrow \alpha}K_{an}\left(
\mathcal{A}_{i},\mathcal{B}_{i}\right)  $

\item $K_{n}\left(  \rho \left(  \mathcal{A}\right)  ,\overset{\alpha}{%
{\displaystyle \prod}
}\mathcal{B}_{i}\right)  \leq \lim_{i\rightarrow \alpha}K_{n}\left(
\mathcal{A}_{i},\mathcal{B}_{i}\right)  $

\item If each $\mathcal{A}_{i}$ is a masa in a von Neumann algebra
$\mathcal{B}_{i}$, then
\[
K_{n}\left(  \pi \left(  \mathcal{A}\right)  ,%
{\displaystyle \prod_{i\in I}}
\mathcal{B}_{i}/\sum_{i\in I}\mathcal{B}_{i}\right)  \leq1\text{ and }%
K_{n}\left(  \rho \left(  \mathcal{A}\right)  ,\overset{\alpha}{%
{\displaystyle \prod}
}\mathcal{B}_{i}\right)  \leq1.
\]

\item If each $\mathcal{B}_{i}$ is primitive or a von Neumann algebra, then%
\[
\mathcal{Z}\left(
{\displaystyle \prod^{\alpha}}
\mathcal{B}_{i}\right)  =%
{\displaystyle \prod^{\alpha}}
\mathcal{Z}\left(  \mathcal{B}_{i}\right)  \text{ and }%
\]%
\[
\mathcal{Z}\left(
{\displaystyle \prod_{i\in I}}
\mathcal{B}_{i}/\sum_{i\in I}\mathcal{B}_{i}\right)  =\pi \left(
{\displaystyle \prod_{i\in I}}
\mathcal{Z}\left(  \mathcal{B}_{i}\right)  \right)
\]

\item If, for each $i\in I$, $\mathcal{B}_{i}=B\left(  H_{i}\right)  $ for
some Hilbert space $H_{i}$, then%
\[
K_{an}\left(  \pi \left(  \mathcal{A}\right)  ,%
{\displaystyle \prod_{i\in I}}
\mathcal{B}_{i}/\sum_{i\in I}\mathcal{B}_{i}\right)  \leq29\text{ and }%
K_{an}\left(  \rho \left(  \mathcal{A}\right)  ,\overset{\alpha}{%
{\displaystyle \prod}
}\mathcal{B}_{i}\right)  \leq29.
\]

\item If each $\mathcal{B}_{i}$ is a von Neumann algebra and $\mathcal{D}$ is
a unital commutative C*-subalgebra of $%
{\displaystyle \prod_{i\in I}}
B_{i}$, then%
\[
K_{an}\left(  C^{\ast}\left(  \pi \left(  \mathcal{D}\right)  \cup
\mathcal{Z}\left(
{\displaystyle \prod_{i\in I}}
\mathcal{B}_{i}/\sum_{i\in I}\mathcal{B}_{i}\right)  \right)  ,%
{\displaystyle \prod_{i\in I}}
\mathcal{B}_{i}/\sum_{i\in I}\mathcal{B}_{i}\right)  \leq4\text{ and }%
\]%
\[
K_{an}\left(  C^{\ast}\left(  \rho \left(  \mathcal{D}\right)  \cup
\mathcal{Z}\left(
{\displaystyle \prod^{\alpha}}
\mathcal{B}_{i}\right)  \right)  ,%
{\displaystyle \prod^{\alpha}}
\mathcal{B}_{i}\right)  \leq4,
\]

\item If $\mathcal{D}\subseteq%
{\displaystyle \prod_{i\in I}}
\mathcal{B}_{i}$ is norm separable and $I=\mathbb{N}$, then
\[
K_{n}\left(  C^{\ast}\left(  \pi \left(  \mathcal{D}\right)  \cup
\mathcal{Z}\left(
{\displaystyle \prod_{i\in I}}
\mathcal{B}_{i}/\sum_{i\in I}\mathcal{B}_{i}\right)  \right)  ,%
{\displaystyle \prod_{i\in I}}
\mathcal{B}_{i}/\sum_{i\in I}\mathcal{B}_{i}\right)  =\text{ }%
\]%
\[
K_{an}\left(  C^{\ast}\left(  \pi \left(  \mathcal{D}\right)  \cup
\mathcal{Z}\left(
{\displaystyle \prod_{i\in I}}
\mathcal{B}_{i}/\sum_{i\in I}\mathcal{B}_{i}\right)  \right)  ,%
{\displaystyle \prod_{i\in I}}
\mathcal{B}_{i}/\sum_{i\in I}\mathcal{B}_{i}\right)  \text{ and}%
\]%
\[
K_{n}\left(  C^{\ast}\left(  \rho \left(  \mathcal{D}\right)  \cup
\mathcal{Z}\left(
{\displaystyle \prod^{\alpha}}
\mathcal{B}_{i}\right)  \right)  ,%
{\displaystyle \prod^{\alpha}}
\mathcal{B}_{i}\right)  =
\]%
\[
K_{an}\left(  C^{\ast}\left(  \rho \left(  \mathcal{D}\right)  \cup
\mathcal{Z}\left(
{\displaystyle \prod^{\alpha}}
\mathcal{B}_{i}\right)  \right)  ,%
{\displaystyle \prod^{\alpha}}
\mathcal{B}_{i}\right)  .
\]

\end{enumerate}
\end{proposition}

\begin{proof}
$\left(  1\right)  $ Let $\Delta=\sup_{i\in I}K_{an}\left(  \mathcal{A}%
_{i},\mathcal{B}_{i}\right)  $. If $\Delta=\infty$, there is nothing to prove,
so we can assume that $0<\Delta<\infty$. Suppose $T=\pi \left(  \left \{
T_{i}\right \}  \right)  \in%
{\displaystyle \prod_{i\in I}}
\mathcal{B}_{i}/\sum_{i\in I}\mathcal{B}_{i}$ and, for $1\leq k\leq m,$
$A_{k}=\left \{  A_{k,i}\right \}  \in \mathcal{A}$ and suppose $\varepsilon>0$.
Then, for each $i\in I,$
\[
dist\left(  T_{i},\mathcal{A}_{i}\right)  \leq K_{an}\left(  \mathcal{A}%
_{i},\mathcal{B}_{i}\right)  d_{an}\left(  T_{i},\mathcal{A}_{i}%
,\mathcal{B}_{i}\right)  \leq \Delta d_{an}\left(  T_{i},\mathcal{A}%
_{i},\mathcal{B}_{i}\right)  .
\]
Hence there is a unitary $U_{i}\in \mathcal{B}_{i}$ such that%
\[
\left \Vert U_{i}A_{k,i}-A_{k,i}U_{i}\right \Vert <\varepsilon
\]
for $1\leq k\leq m$ and such that%
\[
dist\left(  T_{i},\mathcal{A}_{i}\right)  <\Delta \left(  \left \Vert U_{i}%
T_{i}-T_{i}U_{i}\right \Vert +\varepsilon \right)  .
\]
Hence, for each $i\in I$ there is a $C_{i}\in \mathcal{A}_{i}$ such that%
\[
\left \Vert T_{i}-C_{i}\right \Vert <\Delta \left(  \left \Vert U_{i}T_{i}%
-T_{i}U_{i}\right \Vert +\varepsilon \right)
\]
Then $U=\left \{  U_{i}\right \}  \in%
{\displaystyle \prod_{i\in I}}
\mathcal{B}_{i}$ is a unitary and $C=\left \{  C_{i}\right \}  \in \mathcal{A}$.
Moreover%
\[
dist\left(  T,\pi \left(  \mathcal{A}\right)  \right)  \leq \left \Vert
T-\pi \left(  C\right)  \right \Vert =\limsup_{i\rightarrow \infty}\left \Vert
T_{i}-C_{i}\right \Vert \leq
\]%
\[
\limsup_{i\rightarrow \infty}\Delta \left(  \left \Vert U_{i}T_{i}-T_{i}%
U_{i}\right \Vert +\varepsilon \right)  =\Delta \left[  \left \Vert \pi \left(
U\right)  T-T\pi \left(  U\right)  \right \Vert +\varepsilon \right]  ,
\]
and, for $1\leq k\leq m,$%
\[
\left \Vert \pi \left(  U\right)  A_{k}-A_{k}\pi \left(  U\right)  \right \Vert
=\limsup_{i\rightarrow \infty}\left \Vert U_{i}A_{k,i}-A_{k,i}U_{i}\right \Vert
\leq \varepsilon.
\]
If we let $\Lambda$ be the set of all pairs $\lambda=\left(  \mathcal{F}%
,\varepsilon \right)  $ with $\varepsilon>0$ and $\mathcal{F}=\left \{
A_{1},\ldots,A_{m}\right \}  $ a finite subset of $\pi \left(  \mathcal{A}%
\right)  $ and we let $V_{\lambda}=\pi \left(  U\right)  $ constructed above,
then $\left \{  V_{\lambda}\right \}  $ is a net of unitary elements of $%
{\displaystyle \prod_{i\in I}}
\mathcal{B}_{i}/\sum_{i\in I}\mathcal{B}_{i}$ such that, for every $A\in
\pi \left(  \mathcal{A}\right)  $%
\[
\left \Vert V_{\lambda}A-AV_{\lambda}\right \Vert \rightarrow0,
\]
and such that%
\[
dist\left(  T,\pi \left(  \mathcal{A}\right)  \right)  \leq \Delta
\limsup \left \Vert TV_{\lambda}-V_{\lambda}T\right \Vert \leq \Delta
d_{an}\left(  T,\pi \left(  \mathcal{A}\right)  ,%
{\displaystyle \prod_{i\in I}}
\mathcal{B}_{i}/\sum_{i\in I}\mathcal{B}_{i}\right)  .
\]
\bigskip Hence $K_{an}\left(  \pi \left(  \mathcal{A}\right)  ,%
{\displaystyle \prod_{i\in I}}
\mathcal{B}_{i}/\sum_{i\in I}\mathcal{B}_{i}\right)  \leq \Delta$.

$\left(  2\right)  .$ Now we let $\Delta=\sup_{i\in I}K_{n}\left(
\mathcal{A}_{i},\mathcal{B}_{i}\right)  $. Suppose $T=\pi \left(  \left \{
T_{i}\right \}  \right)  \in%
{\displaystyle \prod_{i\in I}}
\mathcal{B}_{i}/\sum_{i\in I}\mathcal{B}_{i}$ and $\varepsilon>0.$ As in the
proof of $\left(  1\right)  $, for each $i\in I$, we can choose a unitary
$U_{i}\in \left(  \mathcal{A}_{i},\mathcal{B}_{i}\right)  ^{\prime}$ and a
$C_{i}\in \mathcal{A}_{i}$ such that%
\[
\left \Vert C_{i}-T_{i}\right \Vert \leq \Delta \left[  \left \Vert U_{i}%
T_{i}-T_{i}U_{i}\right \Vert +\varepsilon \right]  .
\]
Hence $U=\pi \left(  \left \{  U_{i}\right \}  \right)  $ is a unitary in
$\left(  \pi \left(  \mathcal{A}\right)  ,%
{\displaystyle \prod_{i\in I}}
\mathcal{B}_{i}/\sum_{i\in I}\mathcal{B}_{i}\right)  ^{\prime}$ and%
\[
dist\left(  T,\pi \left(  \mathcal{A}\right)  \right)  \leq \limsup
_{i\rightarrow \infty}\left \Vert T_{i}-C_{i}\right \Vert \leq \Delta \left \Vert
UT-TU\right \Vert \leq
\]%
\[
\Delta d_{n}\left(  T,\pi \left(  \mathcal{A}\right)  ,%
{\displaystyle \prod_{i\in I}}
\mathcal{B}_{i}/\sum_{i\in I}\mathcal{B}_{i}\right)  .
\]

$\left(  3\right)  $ and $\left(  4\right)  .$ The proofs are almost the same
as those of $\left(  1\right)  $ and $\left(  2\right)  $.

$\left(  5\right)  .$ This follows from $\left(  2\right)  $ and $\left(
4\right)  .$

$\left(  6\right)  .$ This follows from $\left(  2\right)  $ and $\left(
4\right)  $ and the fact that $\mathcal{Z}\left(  \mathcal{B}\right)  $ is
weakly injective when $\mathcal{B}$ is primitive or a von Neumann algebra,
which implies $K_{n}\left(  \mathcal{Z}\left(  \mathcal{B}\right)
,\mathcal{B}\right)  =1$.

$\left(  7\right)  .$ This follows from $\left(  1\right)  $ and $\left(
3\right)  $ and the fact from \cite{H2} that $K_{an}\left(  \mathcal{C}%
,B\left(  H\right)  \right)  \leq29$ for every Hilbert space $H$ and every
unital C*-subalgebra $\mathcal{C}\subseteq B\left(  H\right)  $.

$\left(  8\right)  $ We can find, for each $i\in I$, a masa $\mathcal{A}_{i}$
in $\mathcal{B}_{i}$ so that $\mathcal{A=}%
{\displaystyle \prod_{i\in I}}
\mathcal{A}_{i}$ contains $\mathcal{D}$. We know from $\left(  5\right)  $
that
\[
K_{n}\left(  \pi \left(  \mathcal{A}\right)  ,%
{\displaystyle \prod_{i\in I}}
\mathcal{B}_{i}/\sum_{i\in I}\mathcal{B}_{i}\right)  \leq1\text{ and }%
K_{n}\left(  \rho \left(  \mathcal{A}\right)  ,\overset{\alpha}{%
{\displaystyle \prod}
}\mathcal{B}_{i}\right)  \leq1.
\]
Suppose $T=\pi \left(  \left \{  A_{i}\right \}  \right)  \in \pi \left(
\mathcal{A}\right)  $ with $A=\left \{  A_{i}\right \}  \in \mathcal{A}$ and
suppose $\varepsilon>0$. We know from $\left(  5\right)  $ that $\mathcal{Z}%
\left(
{\displaystyle \prod_{i\in I}}
\mathcal{B}_{i}/\sum_{i\in I}\mathcal{B}_{i}\right)  =\pi \left(
{\displaystyle \prod_{i\in I}}
\mathcal{Z}\left(  \mathcal{B}_{i}\right)  \right)  $, so $\pi \left(  C^{\ast
}\left(  \mathcal{D}\cup%
{\displaystyle \prod_{i\in I}}
\mathcal{Z}\left(  \mathcal{B}_{i}\right)  \cup \sum^{\oplus}\mathcal{A}%
_{i}\right)  \right)  =C^{\ast}\left(  \pi \left(  \mathcal{D}\right)
\cup \mathcal{Z}\left(
{\displaystyle \prod_{i\in I}}
\mathcal{B}_{i}/\sum_{i\in I}\mathcal{B}_{i}\right)  \right)  $ and thus
$\mathcal{E}=_{\text{\textrm{def}}}C^{\ast}\left(  \mathcal{D}\cup%
{\displaystyle \prod_{i\in I}}
\mathcal{Z}\left(  \mathcal{B}_{i}\right)  \cup \sum^{\oplus}\mathcal{A}%
_{i}\right)  \subseteq \mathcal{A}$. It is clear that $dist\left(  T,C^{\ast
}\left(  \pi \left(  \mathcal{D}\right)  \cup \mathcal{Z}\left(
{\displaystyle \prod_{i\in I}}
\mathcal{B}_{i}/\sum_{i\in I}\mathcal{B}_{i}\right)  \right)  \right)  $ is
the same as $dist\left(  A,\mathcal{E}\right)  $. However, it follows from
Proposition \ref{Macapp} (with $\mathcal{W}=\mathbb{C}$) that there is a net
$\left \{  U_{\lambda}\right \}  $ of unitary elements of $%
{\displaystyle \prod_{i\in I}}
\mathcal{B}_{i}$ such that
\[
\left \Vert U_{\lambda}S-SU_{\lambda}\right \Vert \rightarrow0
\]
for every $S\in \mathcal{E}$ and such that%
\[
dist\left(  A,\mathcal{E}\right)  \leq \lim_{\lambda}\left \Vert U_{\lambda
}A-AU_{\lambda}\right \Vert .
\]
If $J\subseteq I$ and $S=\left \{  S_{i}\right \}  \in$ $%
{\displaystyle \prod_{i\in I}}
\mathcal{B}_{i}$, we define $P_{J}S=\left \{  S_{i}^{\prime}\right \}  \in%
{\displaystyle \prod_{i\in I}}
\mathcal{B}_{i}$, where
\[
S_{i}^{^{\prime}}=\left \{
\begin{array}
[c]{cc}%
S_{i} & \text{if }i\in J\\
0 & \text{if }i\notin J
\end{array}
\right.  .
\]
Since $A\in \mathcal{A}$, it follows, for every finite subset $J\subseteq I$,
that $P_{J}A\in \sum_{i\in I}\mathcal{A}_{i}\subseteq \mathcal{E}$. Hence, for
every finite $J\subseteq I$, we have%
\[
\lim \left \Vert U_{\lambda}A-AU_{\lambda}\right \Vert =\lim \left \Vert
U_{\lambda}P_{I\backslash J}A-P_{I\backslash J}AU_{\lambda}\right \Vert .
\]
Suppose $F\subseteq \mathcal{D}$ is finite and $\varepsilon>0$. We write
$U_{\lambda}=\left \{  U_{\lambda}\left(  i\right)  \right \}  $ and, for each
$D\in F$, we write $D=\left \{  D_{i}\right \}  $. It follows that the set
$I_{\left(  F,\varepsilon \right)  }$ of $i\in I$ for which there is a unitary
$W_{i}\in \mathcal{B}_{i}$ with
\[
\max_{D\in F}\left \Vert W_{i}D_{i}-D_{i}W_{i}\right \Vert <\varepsilon \text{
and}%
\]%
\[
dist\left(  A,\mathcal{E}\right)  \leq \left \Vert W_{i}A_{i}-A_{i}%
W_{i}\right \Vert +\varepsilon
\]
must be infinite. Hence we can choose a unitary $W_{\left(  F,\varepsilon
\right)  }=\left \{  W_{\left(  F,\varepsilon \right)  }\left(  i\right)
\right \}  $ so that%
\[
W_{\left(  F,\varepsilon \right)  }\left(  i\right)  =\left \{
\begin{array}
[c]{cc}%
W_{i} & \text{if }i\in I_{\left(  F,\varepsilon \right)  }\\
1 & \text{otherwise}%
\end{array}
\right.  .
\]
It follows that%
\[
\max_{D\in F}\left \Vert DW_{\left(  F,\varepsilon \right)  }-W_{\left(
F,\varepsilon \right)  }D\right \Vert <\varepsilon
\]
and%
\[
dist\left(  A,\mathcal{E}\right)  \leq \left \Vert \pi \left(  W_{\left(
F,\varepsilon \right)  }A-AW_{\left(  F,\varepsilon \right)  }\right)
\right \Vert =\left \Vert \pi \left(  W_{\left(  F,\varepsilon \right)  }\right)
T-T\pi \left(  W_{\left(  F,\varepsilon \right)  }\right)  \right \Vert .
\]
It follows that $\left \{  \pi \left(  W_{\left(  F,\varepsilon \right)
}\right)  \right \}  $ is a net of unitary elements of $%
{\displaystyle \prod_{i\in I}}
\mathcal{B}_{i}/\sum_{i\in I}\mathcal{B}_{i}$ such that%
\[
\left \Vert \pi \left(  W_{\left(  F,\varepsilon \right)  }\right)  S-S\pi \left(
W_{\left(  F,\varepsilon \right)  }\right)  \right \Vert \rightarrow0
\]
for every $S\in \pi \left(  \mathcal{D}\right)  $ and such that%
\[
dist\left(  T,C^{\ast}\left(  \pi \left(  \mathcal{D}\right)  \cup
\mathcal{Z}\left(
{\displaystyle \prod_{i\in I}}
\mathcal{B}_{i}/\sum_{i\in I}\mathcal{B}_{i}\right)  \right)  \right)
=dist\left(  A,\mathcal{E}\right)  \leq
\]%
\[
\limsup_{\left(  F,\varepsilon \right)  }\left \Vert \pi \left(  W_{\left(
F,\varepsilon \right)  }\right)  T-T\pi \left(  W_{\left(  F,\varepsilon \right)
}\right)  \right \Vert \leq
\]%
\[
d_{an}\left(  T,C^{\ast}\left(  \pi \left(  \mathcal{D}\right)  \cup
\mathcal{Z}\left(
{\displaystyle \prod_{i\in I}}
\mathcal{B}_{i}/\sum_{i\in I}\mathcal{B}_{i}\right)  \right)  ,%
{\displaystyle \prod_{i\in I}}
\mathcal{B}_{i}/\sum_{i\in I}\mathcal{B}_{i}\right)  .
\]
The fact that $K_{n}\left(  \mathcal{A}\right)  \leq1$ (by part $\left(
5\right)  $) implies, reasoning as in the proof of Lemma \ref{rel}, we see
that
\[
K_{an}\left(  C^{\ast}\left(  \pi \left(  \mathcal{D}\right)  \cup
\mathcal{Z}\left(
{\displaystyle \prod_{i\in I}}
\mathcal{B}_{i}/\sum_{i\in I}\mathcal{B}_{i}\right)  \right)  \right)  \leq4.
\]
The argument for ultraproducts is the same except for considering finite
subsets $J\subseteq I$ we consider subsets $J$ not in the ultrafilter $\alpha
$, which shows that $I_{\left(  F,\varepsilon \right)  }\in \alpha$.

$\left(  9\right)  .$ This follows using arguments in the proof of
\cite[Theorem 4]{H4}.
\end{proof}

\bigskip

\bigskip

\bigskip

\section{Representations}

In \cite{AP} C. Akemann and G. Pedersen showed that central sequences from a
quotient $\mathcal{B}/\mathcal{J}$ can be lifted to a central sequence in
$\mathcal{B}$. The ideas in their proof can be used here. Recall from
\cite{AP} and \cite{Arv} that if $\mathcal{J}$ is a closed ideal in a unital
C*-algebra $\mathcal{B}$ there is a \emph{quasicentral approximate unit},
i.e., a net $\left \{  e_{\lambda}\right \}  _{\lambda \in \Lambda}$ in
$\mathcal{J}$ such that

\begin{enumerate}
\item $0\leq e_{\lambda}\leq1$ for every $\lambda \in \Lambda$,

\item $\left \Vert \left(  1-e_{\lambda}\right)  x\right \Vert +\left \Vert
x\left(  1-e_{\lambda}\right)  \right \Vert \rightarrow0$ for every
$x\in \mathcal{J}$,

\item $\left \Vert be_{\lambda}-e_{\lambda}b\right \Vert \rightarrow0$ for every
$b\in \mathcal{B}$.
\end{enumerate}

\bigskip

It is well-known \cite{AP} that if $\pi:\mathcal{B\ }\rightarrow
\mathcal{B}/\mathcal{J}$ is the quotient homomorphism, then $\left \Vert
\left(  1-e_{\lambda}\right)  b\right \Vert \rightarrow \left \Vert \pi \left(
b\right)  \right \Vert $ for every $b\in \mathcal{B}$.

\begin{theorem}
\label{lift}Suppose $\mathcal{B}$ and $\mathcal{E}$ are unital C*-algebras and
$\pi:\mathcal{B}\rightarrow \mathcal{E}$ is a unital surjective $\ast
$-homomorphism. If $\mathcal{S}\subseteq \mathcal{B}$, then%
\[
\pi \left(  Appr\left(  \mathcal{S},\mathcal{B}\right)  ^{\prime \prime}\right)
\subseteq Appr\left(  \pi \left(  S\right)  ,\mathcal{E}\right)  ^{\prime
\prime}.
\]

\end{theorem}

\begin{proof}
Let $\left \{  e_{\lambda}\right \}  _{\lambda \in \Lambda}$ be a quasicentral
approximate unit for $\ker \pi$. Then, for every $x,y\in \mathcal{B}$,%
\[
\left \Vert \left(  1-e_{\lambda}\right)  x\right \Vert \rightarrow \left \Vert
\pi \left(  x\right)  \right \Vert ,
\]
and%
\[
\left \Vert \left[  \left(  1-e_{\lambda}\right)  x\right]  y-y\left[  \left(
1-e_{\lambda}\right)  x\right]  \right \Vert \rightarrow \left \Vert \pi \left(
x\right)  \pi \left(  y\right)  -\pi \left(  y\right)  \pi \left(  x\right)
\right \Vert .
\]
The second follows from the first statement and
\[
\left \Vert y\left(  1-e_{\lambda}\right)  -\left(  1-e_{\lambda}\right)
y\right \Vert \rightarrow0.
\]
Suppose $x\in \mathcal{B}$ and $\pi \left(  x\right)  \notin Appr\left(
\pi \left(  S\right)  ,\mathcal{E}\right)  ^{\prime \prime}$. Then there is an
$\varepsilon>0$ such that for every finite subset $\mathcal{F}$ of
$\mathcal{S}$ and every $\eta>0$ there is a $y\in \mathcal{B}$ such that%
\[
\left \Vert \pi \left(  y\right)  \right \Vert <1
\]%
\[
\left \Vert \pi \left(  y\right)  \pi \left(  w\right)  -\pi \left(  w\right)
\pi \left(  y\right)  \right \Vert <\eta
\]
for every $w\in \mathcal{S}$ and
\[
\left \Vert \pi \left(  y\right)  \pi \left(  x\right)  -\pi \left(  x\right)
\pi \left(  y\right)  \right \Vert >\varepsilon \text{.}%
\]
It follows from the above remarks that there is a $\lambda \in \Lambda$ such
that if $y_{\left(  \mathcal{F},\eta \right)  }=\left(  1-e_{\lambda}\right)
y$, then%
\[
\left \Vert y_{\left(  \mathcal{F},\eta \right)  }\right \Vert <1,
\]%
\[
\left \Vert y_{\left(  \mathcal{F},\eta \right)  }w-wy_{\left(  \mathcal{F}%
,\eta \right)  }\right \Vert <\eta
\]
for every $w\in \mathcal{S}$, and%
\[
\left \Vert y_{\left(  \mathcal{F},\eta \right)  }x-xy_{\left(  \mathcal{F}%
,\eta \right)  }\right \Vert >\varepsilon.
\]
Then $\left \{  y_{\left(  \mathcal{F},\eta \right)  }\right \}  $ is a bounded
net such that $\left \Vert y_{\left(  \mathcal{F},\eta \right)  }w-wy_{\left(
\mathcal{F},\eta \right)  }\right \Vert \rightarrow0$ for every $w\in
\mathcal{S}$ and such that $\left \Vert y_{\left(  \mathcal{F},\eta \right)
}x-xy_{\left(  \mathcal{F},\eta \right)  }\right \Vert \nrightarrow0$. Hence
$x\notin Appr\left(  \mathcal{S},\mathcal{B}\right)  ^{\prime \prime}$.
\end{proof}

\bigskip \bigskip

It is easy to show that a direct product of unital centrally prime C*-algebras
is centrally prime. The following result shows that the same is not true for
subdirect products. This gives a way to construct examples of commutative
unital C*-subalgebras of a C*-algebra $\mathcal{B}$ for which $Appr\left(
\mathcal{A},\mathcal{B}\right)  ^{\prime \prime}$ is much larger than $C^{\ast
}\left(  \mathcal{A}\cup \mathcal{Z}\left(  \mathcal{B}\right)  \right)  .$
Note that in the following lemma the algebra $\mathcal{A}$ is not assumed to
be selfadjoint.

\bigskip

\begin{lemma}
\label{subdirect}Suppose $\mathcal{B}_{1}$, $\mathcal{B}_{2},\ldots
,\mathcal{B}_{n}$ are unital C*-algebras and $\mathcal{B}\subseteq
\mathcal{B}_{1}\oplus \mathcal{B}_{2}\oplus \cdots \oplus \mathcal{B}_{n}$ is a
unital C*-algebra such that the coordinate projection $\pi_{j}:\mathcal{B}%
\rightarrow \mathcal{B}_{j}$ is surjective for $j=1,2,\ldots,n$. Then, for
every unital norm closed subalgebra $\mathcal{A}$ of $\mathcal{B}$, we have%
\[
Appr\left(  \mathcal{A},\mathcal{B}\right)  ^{\prime \prime}=\left[
Appr\left(  \pi_{1}\left(  \mathcal{A}\right)  ,\mathcal{B}_{1}\right)
^{\prime \prime}\oplus \cdots \oplus Appr\left(  \pi_{n}\left(  \mathcal{A}%
\right)  ,\mathcal{B}_{n}\right)  ^{\prime \prime}\right]  \cap \mathcal{B}.
\]
\bigskip
\end{lemma}

\begin{proof}
It follows from Theorem \ref{lift} and the surjectivity of $\pi_{j}$ that
\[
\pi_{j}\left(  Appr\left(  \mathcal{A},\mathcal{B}\right)  ^{\prime \prime
}\right)  \subseteq Appr\left(  \pi_{j}\left(  \mathcal{A}\right)
,\mathcal{B}_{j}\right)  ^{\prime \prime}%
\]
for $j=1,2,\ldots,n$. Hence%
\[
Appr\left(  \mathcal{A},\mathcal{B}\right)  ^{\prime \prime}\subseteq \left[
Appr\left(  \pi_{1}\left(  \mathcal{A}\right)  ,\mathcal{B}_{1}\right)
^{\prime \prime}\oplus \cdots \oplus Appr\left(  \pi_{n}\left(  \mathcal{A}%
\right)  ,\mathcal{B}_{n}\right)  ^{\prime \prime}\right]  \cap \mathcal{B}%
\text{.}%
\]
Next suppose $b_{j}\in Appr\left(  \pi_{j}\left(  \mathcal{A}\right)
,\mathcal{B}_{j}\right)  ^{\prime \prime}$ for $j=1,2,\ldots,n$ and
$b=b_{1}\oplus b_{2}\oplus \cdots \oplus b_{n}\in \mathcal{B}$. Suppose $\left \{
x_{\lambda}=x_{\lambda,1}\oplus x_{\lambda,2}\oplus \cdots \oplus x_{\lambda
,n}\right \}  $ is a bounded net in $\mathcal{B}$ such that, for every
$a=\pi_{1}\left(  a\right)  \oplus \pi_{2}\left(  a\right)  \oplus \cdots
\oplus \pi_{n}\left(  a\right)  \in \mathcal{A}$,
\[
\left \Vert ax_{\lambda}-x_{\lambda}a\right \Vert \rightarrow0.
\]
Then
\[
\left \Vert \pi_{j}\left(  a\right)  x_{\lambda,j}-x_{\lambda,j}\pi_{j}\left(
a\right)  \right \Vert \rightarrow0\text{ for }1\leq j\leq n.
\]
Hence, for $1\leq j\leq n$, $\left \{  x_{\lambda,j}\right \}  $ is a bounded
net in $\mathcal{B}_{j}$ such that, for every $c\in \pi \left(  \mathcal{A}%
_{j}\right)  $%
\[
\left \Vert x_{\lambda,j}c-cx_{\lambda,j}\right \Vert \rightarrow0.
\]
Hence%
\[
\left \Vert x_{\lambda}b-bx_{\lambda}\right \Vert =\left \Vert \left(
b_{1}x_{\lambda,1}-x_{\lambda,1}b_{1}\right)  \oplus \cdots \oplus \left(
b_{n}x_{\lambda,n}-x_{\lambda,n}b_{n}\right)  \right \Vert \rightarrow0.
\]
Hence $b\in Appr\left(  \mathcal{A},\mathcal{B}\right)  ^{\prime \prime}$.
\bigskip
\end{proof}

\begin{corollary}
Suppose $\mathcal{B}_{1}$, $\mathcal{B}_{2},\ldots,\mathcal{B}_{n}$ are unital
centrally prime C*-algebras and $\mathcal{B}\subseteq \mathcal{B}_{1}%
\oplus \cdots \oplus \mathcal{B}_{n}$ is a unital C*-algebra such that the
coordinate projection $\pi_{j}:\mathcal{B}\rightarrow \mathcal{B}_{j}$ is
surjective for $j=1,2,\ldots,n$. Then, for every unital commutative
C*-subalgebra $\mathcal{A}$ of $\mathcal{B}$, we have%
\[
Appr\left(  \mathcal{A},\mathcal{B}\right)  ^{\prime \prime}=\left[  C^{\ast
}\left(  \pi_{1}\left(  \mathcal{A}\right)  \cup \mathcal{Z}\left(
\mathcal{B}_{1}\right)  \right)  \oplus \cdots \oplus C^{\ast}\left(  \pi
_{n}\left(  \mathcal{A}\right)  \cup \mathcal{Z}\left(  \mathcal{B}_{n}\right)
\right)  \right]  \cap \mathcal{B}%
\]

\end{corollary}

\begin{example}
Let $S$ be the unilateral shift operator on $\ell^{2}$, and let $\mathcal{B}%
=C^{\ast}\left(  S^{\ast}\oplus S\right)  .$ It follows that $\mathcal{K}%
\left(  \ell^{2}\right)  \oplus \mathcal{K}\left(  \ell^{2}\right)
\subseteq \mathcal{B}\neq C^{\ast}\left(  \mathcal{S}^{\ast}\right)  \oplus
C^{\ast}\left(  S\right)  $ and $\mathcal{Z}\left(  \mathcal{B}\right)
=\mathbb{C}1\subseteq \mathcal{A}$. If $0\neq A=A^{\ast}\in \mathcal{K}\left(
\ell^{2}\right)  $ and $\mathcal{A}=C^{\ast}\left(  A\oplus A\right)  ,$ then
$\mathcal{A}$ a unital commutative C*-subalgebra of $\mathcal{B}$,
$\mathcal{Z}\left(  \mathcal{B}\right)  \subseteq \mathcal{A}$, but
\[
Appr\left(  \mathcal{A},\mathcal{B}\right)  ^{\prime \prime}=C^{\ast}\left(
A\right)  \oplus C^{\ast}\left(  A\right)  ,
\]
which is much larger than $\mathcal{A}$.
\end{example}

\section{C*-algebraic Stone-Weierstrass and Continuous Fields}

Here is our main result in this section. The proof is based on the factor
state version of the Stone-Weierstrass theorem of Longo \cite{L}, Popa
\cite{P},(and Teleman \cite{T}).

\begin{theorem}
\label{SWapp}Suppose $\mathcal{B}$ is a unital separable C*-algebra and
$\mathcal{A}$ is a unital C*-subalgebra of $\mathcal{B}$ with $\mathcal{Z}%
\left(  \mathcal{B}\right)  \subseteq \mathcal{A}$. Suppose also $\left \{
\mathcal{J}_{i}:i\in I\right \}  $ is a family of closed two-sided ideals of
$\mathcal{B}$ such that

\begin{enumerate}
\item If $i\neq j$ are in $I$, then%
\[
\left(  \mathcal{A}\cap \mathcal{J}_{i}\right)  +\left(  \mathcal{A}%
\cap \mathcal{J}_{j}\right)  =\mathcal{A}%
\]

\item $\mathcal{A}/\left(  \mathcal{A}\cap \mathcal{J}_{i}\right)  $ is
approximately normal in $\mathcal{B}/\mathcal{J}_{i}$ for each $i\in I$.

\item If $\mathcal{J}$ is a primitive ideal in $\mathcal{B}$, then there is an
$i\in I$ such that $\mathcal{J}\subseteq \mathcal{J}_{i}$.
\end{enumerate}

Then $\mathcal{A}$ is approximately normal in $\mathcal{B}$.
\end{theorem}

\begin{proof}
Assume via contradiction that $T\in Appr\left(  \mathcal{A},\mathcal{B}%
\right)  ^{\prime \prime}$ and $T\notin \mathcal{A}$. It follows from the factor
state Stone-Weierstrass theorem \cite{L}, \cite{P}, that there are factor
states $\alpha \neq \beta$ on $C^{\ast}\left(  \mathcal{A\cup}\left \{
T\right \}  \right)  $ such that $\alpha \left(  A\right)  =\beta \left(
A\right)  $ for every $A\in \mathcal{A}$. We can choose $S\in C^{\ast}\left(
\mathcal{A}\cup \left \{  T\right \}  \right)  $ so that $\alpha \left(  S\right)
\neq \beta \left(  S\right)  $. Since $Appr\left(  \mathcal{A},\mathcal{B}%
\right)  ^{\prime \prime}$ is a C*-algebra containing $\mathcal{A}\cup \left \{
T\right \}  $, we see that $S\in Appr\left(  \mathcal{A},\mathcal{B}\right)  $.
It follows from Longo's extension theorem \cite{L} that we can extend $\alpha$
and $\beta$ to factor states on $\mathcal{B}$. Let $\left(  \pi_{\alpha
},H_{\alpha},e_{\alpha}\right)  $ and $\left(  \pi_{\beta},H_{\beta},e_{\beta
}\right)  $ be the GNS representations for $\alpha$ and $\beta$, respectively.
Since $\alpha$ and $\beta$ are factor states, $\pi_{\alpha}\left(
\mathcal{B}\right)  ^{\prime \prime}$ and $\pi_{\beta}\left(  \mathcal{B}%
\right)  ^{\prime \prime}$ are factor von Neumann algebras; whence $\ker
\pi_{\alpha}$ and $\ker \pi_{\beta}$ are prime ideals, which by \cite{Dix} are
primitive. Hence there are $i,j\in I$ such that $\mathcal{J}_{i}\subseteq
\ker \pi_{\alpha}$ and $\mathcal{J}_{j}\subseteq \ker \pi_{\beta}$.

\textbf{Case 1.} $i=j.$ Define $\rho_{i}:\mathcal{B}\rightarrow \mathcal{B}%
/\mathcal{J}_{i}$ to be the quotient homomorphism. It follows that $\rho
_{i}\left(  Appr\left(  \mathcal{A},\mathcal{B}\right)  ^{\prime \prime
}\right)  \subseteq Appr\left(  \rho_{i}\left(  \mathcal{A}\right)  ,\rho
_{i}\left(  \mathcal{B}\right)  \right)  =\rho_{i}\left(  \mathcal{A}\right)
$ since $\rho_{i}\left(  \mathcal{A}\right)  =\mathcal{A}/\left(
\mathcal{A\cap J}_{i}\right)  $ is approximately normal in $\rho_{i}\left(
\mathcal{B}\right)  $. It follows that $\rho_{i}\left(  S\right)  \in \rho
_{i}\left(  \mathcal{A}\right)  $, so there is an $A\in \mathcal{A}$ such that
$S-A\in \ker \rho_{i}=\mathcal{J}_{i}$. But $\mathcal{J}_{i}\subseteq \ker
\pi_{\alpha}$ and $\mathcal{J}_{i}=\mathcal{J}_{j}\subseteq \ker \pi_{\beta}$.
Hence $\pi_{\alpha}\left(  S\right)  =\pi_{\alpha}\left(  A\right)  $ and
$\pi_{\beta}\left(  S\right)  =\pi_{\beta}\left(  A\right)  ,$ which implies
$\alpha \left(  S\right)  =\alpha \left(  A\right)  =\beta \left(  A\right)
=\beta \left(  S\right)  ,$ a contradiction. Hence this case is impossible.

\textbf{Case 2.} $i\neq j$. It follows from assumption $\left(  2\right)  $
that $\left(  \rho_{i}\oplus \rho_{j}\right)  \left(  \mathcal{A}\right)
=\rho_{i}\left(  \mathcal{A}\right)  \oplus \rho_{j}\left(  \mathcal{A}\right)
$. It follows that $\left(  \rho_{i}\oplus \rho_{j}\right)  \left(
\mathcal{B}\right)  =\rho_{i}\left(  \mathcal{B}\right)  \oplus \rho_{j}\left(
\mathcal{B}\right)  $, and we know from Theorem \ref{lift} that
\[
\left(  \rho_{i}\oplus \rho_{j}\right)  \left(  S\right)  \in \left(  \rho
_{i}\oplus \rho_{j}\right)  \left(  Appr\left(  \mathcal{A},\mathcal{B}\right)
^{\prime \prime}\right)  \subseteq Appr\left(  \rho_{i}\left(  \mathcal{A}%
\right)  \oplus \rho_{j}\left(  \mathcal{A}\right)  ,\rho_{i}\left(
\mathcal{B}\right)  \oplus \rho_{j}\left(  \mathcal{B}\right)  \right)  =
\]%
\[
Appr\left(  \rho_{i}\left(  \mathcal{A}\right)  ,\rho_{i}\left(
\mathcal{B}\right)  \right)  ^{\prime \prime}\oplus Appr\left(  \rho_{j}\left(
\mathcal{A}\right)  ,\rho_{j}\left(  \mathcal{B}\right)  \right)
^{\prime \prime}=\rho_{i}\left(  \mathcal{A}\right)  \oplus \rho_{j}\left(
\mathcal{A}\right)  =
\]%
\[
\left(  \rho_{i}\oplus \rho_{j}\right)  \left(  \mathcal{A}\right)  .
\]
Hence there is an $A\in \mathcal{A}$ such that
\[
S-A\in \ker \rho_{i}\cap \ker \rho_{j}\subseteq \ker \pi_{\alpha}\cap \ker \pi_{\beta
}.
\]
Hence,
\[
\alpha \left(  S\right)  =\alpha \left(  A\right)  =\beta \left(  A\right)
=\beta \left(  S\right)  ,
\]
which is also a contradiction.

Since Cases 1 and 2 are both impossible, our assumption that $\mathcal{A}$ is
not approximately normal must be false. This completes the proof. \bigskip
\end{proof}

\begin{corollary}
If in Theorem \ref{SWapp} we replace condition $\left(  3\right)  $ with any
one of

\begin{enumerate}
\item $\mathcal{A}$ is commutative, $\mathcal{Z}\left(  \mathcal{B}%
/\mathcal{J}_{i}\right)  \subseteq \mathcal{A}/\left(  \mathcal{A}%
\cap \mathcal{J}_{i}\right)  $ and $\mathcal{B}/\mathcal{J}_{i}$ is centrally
prime for every $i\in I$,

\item $\mathcal{A=C}^{\ast}\left(  \mathcal{A}_{0}\cup \mathcal{Z}\left(
\mathcal{B}\right)  \right)  $ where $\mathcal{A}_{0}$ is an AF algebra and
each $\mathcal{J}_{i}$ is a primitive ideal,

\item $\mathcal{A=C}^{\ast}\left(  \mathcal{A}_{0}\cup \mathcal{Z}\left(
\mathcal{B}\right)  \right)  $ where $\mathcal{A}_{0}$ is an AH algebra and
each $\mathcal{B}/\mathcal{J}_{i}$ is a von Neumann algebra,

\item $\mathcal{Z}\left(  \mathcal{B}\right)  \subseteq \mathcal{A}$ and each
$\mathcal{B}/\mathcal{J}_{i}$ is finite-dimensional
\end{enumerate}

then $\mathcal{A}$ is approximately normal in $\mathcal{B}$.\bigskip
\end{corollary}

\bigskip

\begin{corollary}
Suppose $\mathcal{D}$ is a separable unital commutative C*-algebra and
$\mathcal{W}$ is a unital C*-algebra, and $\mathcal{A}_{0}$ is a C*-subalgebra
of $\mathcal{B}=\mathcal{D}\otimes \mathcal{W}$. If any one of the following holds,

\begin{enumerate}
\item $\mathcal{A}_{0}$ is commutative and $\mathcal{W}$ is centrally prime,

\item $\mathcal{A}_{0}$ is AF and $\mathcal{W}$ is primitive,

\item $\mathcal{A}_{0}$ is AH and $\mathcal{W}$ is a von Neumann algebra,
\end{enumerate}

then%
\[
Appr\left(  \mathcal{A}_{0},\mathcal{B}\right)  ^{\prime \prime}=C^{\ast
}\left(  \mathcal{A}_{0}\cup \mathcal{Z}\left(  \mathcal{B}\right)  \right)  .
\]

\end{corollary}

\bigskip

\section{C*-algebraic Bishop-Stone-Weierstrass and Nonselfadjoint Subalgebras
\bigskip}

In this section we prove a modest result that applies to commutative
nonselfadjoint subalgebras. The proof relies on the first author's version of
the Bishop-Stone-Weierstrass theorem for C*-algebras \cite{H3}. Suppose
$\mathcal{A}$ is a unital closed (not necessarily selfadjoint) subalgebra of a
unital C*-algebra $\mathcal{B}$. A set $\mathcal{E}$ of states on
$\mathcal{B}$ is called $\mathcal{A}$\emph{-antisymmetric} if whenever
$a\in \mathcal{A}$ and $a|_{\mathcal{E}}$ is real (i.e., $\varphi \left(
a\right)  \in \mathbb{R}$ for all $\varphi$ in $\mathcal{E}$), we have
$a|_{\mathcal{E}}$ is constant. Here is the first author's
Bishop-Stone-Weierstrass theorem for C*-algebras \cite{H3}.

\begin{theorem}
\label{BSW} \cite{H3} Suppose $\mathcal{A}$ is a separable commutative unital
closed subalgebra of a unital C*-algebra $\mathcal{B}$ and $b\in \mathcal{B}$
and suppose for every $\mathcal{A}$-antisymmetric set of pure states on
$\mathcal{B}$ there is an $a\in \mathcal{A}$ such that $b|_{E}=a|_{E}$. Then
$b\in \mathcal{A}$.
\end{theorem}

\bigskip

\begin{theorem}
\label{nsa}Suppose $\mathcal{B}$ is a unital separable C*-algebra
$\mathcal{A}$ is a unital commutative norm-closed subalgebra of $\mathcal{B}$
with $\mathcal{Z}\left(  \mathcal{B}\right)  \subseteq \mathcal{A}$. Suppose
also $\left \{  \mathcal{J}_{i}:i\in I\right \}  $ is a family of closed
two-sided ideals of $\mathcal{B}$ such that

\begin{enumerate}
\item If $i\neq j$ are in $I$, then%
\[
\left(  \mathcal{Z}\left(  \mathcal{B}\right)  \cap \mathcal{J}_{i}\right)
+\left(  \mathcal{Z}\left(  \mathcal{B}\right)  \cap \mathcal{J}_{j}\right)
=\mathcal{Z}\left(  \mathcal{B}\right)
\]

\item $\mathcal{A}/\left(  \mathcal{A}\cap \mathcal{J}_{i}\right)  $ is
approximately normal in $\mathcal{B}/\mathcal{J}_{i}$ for each $i\in I$.

\item If $\mathcal{J}$ is a primitive ideal in $\mathcal{B}$, then there is an
$i\in I$ such that $\mathcal{J}_{i}\subseteq \mathcal{J}$.
\end{enumerate}

Then $\mathcal{A}$ is approximately normal in $\mathcal{B}$.
\end{theorem}

\begin{proof}
Suppose $E$ is an $\mathcal{A}$-antisymmetric set of pure states on
$\mathcal{B}$. Since $\mathcal{Z}\left(  \mathcal{B}\right)  =\mathcal{Z}%
\left(  \mathcal{B}\right)  ^{\ast}\subseteq \mathcal{A}$, it follows that each
element of $\mathcal{Z}\left(  \mathcal{B}\right)  $ is constant on $E$.
Suppose, for $k=1,2,$ that $\alpha_{k}\in E$ with GNS representation $\pi_{k}$
and, by $\left(  3\right)  $, choose $i_{k}\in I$ so that $\mathcal{J}_{i_{k}%
}\subseteq \ker \pi_{k}$. If $i_{1}\neq i_{2}$, it follows from that there is an
$x\in \mathcal{Z}\left(  \mathcal{B}\right)  $ such that $x-1\in \mathcal{J}%
_{i_{1}}$ and $x\in \mathcal{J}_{i_{2}},$ which implies $\pi_{1}\left(
x\right)  =1$ and $\pi_{2}\left(  x\right)  =0$, contradicting $\alpha
_{1}\left(  x\right)  =\alpha_{2}\left(  x\right)  $. Hence there is an $i\in
I$ such that, for every $\alpha \in E$ with GNS representation $\pi$, we have
$\mathcal{J}_{i}\subseteq \ker \pi$. Let $\rho:\mathcal{B\rightarrow
B}/\mathcal{J}_{i}$ be the quotient map. We know from Theorem \ref{lift} that
$\rho \left(  T\right)  \in Appr\left(  \rho \left(  \mathcal{A}\right)
,\rho \left(  \mathcal{B}\right)  \right)  ^{\prime \prime}$. However, it
follows from $\left(  2\right)  $ that $Appr\left(  \rho \left(  \mathcal{A}%
\right)  ,\rho \left(  \mathcal{B}\right)  \right)  ^{\prime \prime}=\rho \left(
\mathcal{A}\right)  $. Hence there is an $A\in \mathcal{A}$ such that
$T-A\in \ker \rho=\mathcal{J}_{i}$. Hence, for every $\alpha \in E$,
$\alpha \left(  T\right)  =\alpha \left(  A\right)  $. It follows from Theorem
\ref{BSW} that $T\in \mathcal{A}$.
\end{proof}

One example of an algebra $\mathcal{B}$ with a family of ideals satisfying
$\left(  1\right)  $ and $\left(  3\right)  $ in Theorem \ref{nsa} is by
letting $\mathcal{B}=C\left(  X\right)  \otimes \mathcal{W=C}\left(
X,\mathcal{W}\right)  $ for some unital C*-algebra $\mathcal{W}$ and some
compact Hausdorff space $X$, and, for each $i\in X$, letting $\mathcal{J}%
_{i}=\left \{  f\in C\left(  X,\mathcal{W}\right)  :f\left(  i\right)
=0\right \}  $. The trick is guaranteeing condition $\left(  2\right)  $.

In \cite{Tur} T. Rolf Turner proved that if $T$ is an algebraic operator on a
Hilbert space $H$, then $\left(  \left \{  T\right \}  ,B\left(  H\right)
\right)  ^{\prime \prime}=\left \{  p\left(  T\right)  :p\in \mathbb{C}\left[
z\right]  \right \}  $. This leads to the first statement in the following lemma.

\bigskip

\begin{lemma}
\label{Little}Suppose $n\in \mathbb{N}$. Then

\begin{enumerate}
\item If $T\in M_{n}\left(  \mathbb{C}\right)  ,$ then the algebra of
polynomials in $T$ is normal.

\item If $n\geq2,$ the following are equivalent:

\begin{enumerate}
\item $n\in \left \{  2,3\right \}  $.

\item Every unital commutative subalgebra of $\mathcal{M}_{n}\left(
\mathbb{C}\right)  $ is normal.
\end{enumerate}
\end{enumerate}
\end{lemma}

\begin{proof}
$\left(  1\right)  .$ This follows from Turner's result \cite{H4}.

$\left(  2\right)  .$ $\left(  a\right)  \Longrightarrow \left(  b\right)  $.
First suppose $n=2$. It follows from Wedderburn's theorem that any commutative
algebra $\mathcal{A}\subseteq \mathcal{M}_{2}\left(  \mathbb{C}\right)  $ is
upper triangular with respect to some basis for $\mathbb{C}^{2}$; whence
$\dim \mathcal{A}$ is at most $2$. This means that there is a $T\in
\mathcal{M}_{2}\left(  \mathbb{C}\right)  $ such that $\mathcal{A}$ is the set
of polynomials in $T$; whence, by $\left(  1\right)  $ above, $\mathcal{A}$ is normal..

Next suppose $n=3$ and $\mathcal{A}$ is a commutative unital subalgebra of
$\mathcal{M}_{3}\left(  \mathbb{C}\right)  $. If $\mathcal{A}$ contains a
nontrivial idempotent, then $\mathcal{A}$ is the direct sum of a subalgebra of
$\mathcal{M}_{2}\left(  \mathbb{C}\right)  $ and $\mathcal{M}_{1}\left(
\mathbb{C}\right)  $, and the desired conclusion follows from the case $n=2$.
If $\mathcal{A}$ contains no nontrivial idempotents, then every element of
$\mathcal{A}$ is the sum of a nilpotent and a scalar multiple of the identity.
Since the algebra generated by a $3\times3$ nilpotent of order $3$ is maximal
abelian, the desired conclusion follows from $\left(  1\right)  $ above
whenever $\mathcal{A}$ contains a nilpotent of order $3$. If the subalgebra
$\mathcal{N}$ of nilpotents in $\mathcal{A}$ has dimension $1$ then the
desired conclusion follows from $\left(  1\right)  $. Since $\mathcal{N}$ is
commutative and is unitarily equivalent to a subalgebra of the strictly
upper-triangular $3\times3$ matrices, we conclude that $\dim \mathcal{N}=2$.
Moreover, every nonzero element of $\mathcal{N}$ is a nilpotent of order $2,$
and therefore has rank $1$. A linear space of rank-one operators must have all
have the form $e\otimes x$ with $e$ fixed or with $x$ fixed and $\left \langle
e,x\right \rangle =0$ (see, for example, \cite[Lemma 4.2]{H4}). Here
\[
\left(  e\otimes x\right)  \left(  h\right)  =\left \langle h,x\right \rangle
e.
\]
Hence $\mathcal{N}$ is unitarily equivalent to
\[
\mathcal{N}_{1}=\left \{  \left(
\begin{array}
[c]{ccc}%
a & b & c\\
0 & a & 0\\
0 & 0 & a
\end{array}
\right)  :a,b,c\in \mathbb{C}\right \}
\]
or%
\[
\mathcal{N}_{2}=\left \{  \left(
\begin{array}
[c]{ccc}%
a & 0 & c\\
0 & a & b\\
0 & 0 & a
\end{array}
\right)  :a,b,c\in \mathbb{C}\right \}  ,
\]
and it is easily shown that $\mathcal{N}_{j}^{\prime}=\mathcal{N}_{j}$ for
$j=1,2$. Hence $\mathcal{N}$ is normal.
\end{proof}

The algebra $\mathcal{A}$ of $4\times4$ matrices of the form $\left(
\begin{array}
[c]{cc}%
\alpha,I_{2} & A\\
0 & \alpha I_{2}%
\end{array}
\right)  ,$ where $\alpha \in \mathbb{C}$ and $A\in \mathcal{M}_{2}\left(
\mathbb{C}\right)  $ and $trace\left(  A\right)  =0$, is commutative and not
normal, since $\left(  \mathcal{A},\mathcal{M}_{4}\left(  \mathbb{C}\right)
\right)  ^{\prime \prime}$ is the set of $4\times4$ matrices of the same form
without the restriction $trace\left(  A\right)  =0.$\bigskip

The following result is an immediate consequence of Theorem \ref{nsa} and
Lemma \ref{Little}.

\bigskip

\begin{theorem}
\label{small} Suppose $K$ is a compact metric space and $\mathcal{B}=C\left(
K\right)  \otimes \mathcal{M}_{n}\left(  \mathbb{C}\right)  $. Then

\begin{enumerate}
\item If $T\in \mathcal{B}$, then the norm closed algebra $\mathcal{A}$
generated by $\left \{  T\right \}  \cup \mathcal{Z}\left(  \mathcal{B}\right)  $
is approximately normal, i.e.,%
\[
Appr\left(  \left \{  T\right \}  ,\mathcal{B}\right)  ^{\prime \prime
}=\mathcal{A}\mathbf{.}%
\]

\item If $n=2$ or $n=3$, then every unital commutative closed subalgebra
$\mathcal{A}$ of $\mathcal{B}$ that contains $\mathcal{Z}\left(
\mathcal{B}\right)  $ is approximately normal, i.e., if $\mathcal{S\subseteq
B}$ is a commuting family, then $Appr\left(  \mathcal{S},\mathcal{B}\right)
^{\prime \prime}$ is the norm closed algebra generated by $\mathcal{S}%
\cup \mathcal{Z}\left(  \mathcal{B}\right)  $.
\end{enumerate}
\end{theorem}

.

\section{Questions and Comments}

We conclude with a list of questions and comments.

\begin{enumerate}
\item If $\mathcal{B}$ is any unital C*-algebra, it is clear that
$\mathcal{Z}\left(  \mathcal{B}\right)  $ is normal. When is $\mathcal{Z}%
\left(  \mathcal{B}\right)  $ metric normal or metric approximately normal? It
is clear that for $T\in \mathcal{B}$, the inner derivation $\delta_{T}$ on
$\mathcal{B}$ defined by $\delta_{T}\left(  S\right)  =TS-ST$ extends to a
weak*-continuous operator on $\mathcal{B}^{\# \#}$, and since the closed unit
ball of $\mathcal{B}$ is weak*-dense in the closed unit ball of $\mathcal{B}%
^{\# \#}$, it follows that%
\[
\left \Vert \delta_{T}\right \Vert =\left \Vert \delta_{T}|_{\mathcal{B}^{\# \#}%
}\right \Vert =2dist\left(  T,\mathcal{Z}\left(  \mathcal{B}^{\# \#}\right)
\right)  .
\]
On the other hand $\left \Vert \delta_{T}\right \Vert $ is clearly equal to
$d_{n}\left(  T,\mathcal{Z}\left(  B\right)  ,\mathcal{B}\right)  $. Hence,
for every $T\in \mathcal{B}$,
\[
dist\left(  T,\mathcal{Z}\left(  \mathcal{B}\right)  \right)  \leq
2K_{n}\left(  \mathcal{Z}\left(  \mathcal{B}\right)  ,\mathcal{B}\right)
dist\left(  T,\mathcal{Z}\left(  \mathcal{B}^{\# \#}\right)  \right)  .
\]
The same argument applies if we replace $\mathcal{B}^{\# \#}$ with $\pi \left(
\mathcal{B}\right)  ^{\prime \prime}$, where $\pi:\mathcal{B}\rightarrow
B\left(  H\right)  $ is a faithful representation. This makes it easy to see
that if $\mathcal{B}$ is primitive, there is a faithful irreducible
representation $\pi$, so%
\[
d_{n}\left(  T,\mathcal{Z}\left(  B\right)  ,\mathcal{B}\right)  =\left \Vert
\delta_{\pi \left(  T\right)  }|_{\pi \left(  \mathcal{B}\right)  ^{\prime
\prime}}\right \Vert =2dist\left(  \pi \left(  T\right)  ,\mathcal{Z}\left(
\pi \left(  \mathcal{B}\right)  ^{\prime \prime}\right)  \right)  =
\]%
\[
2dist\left(  \pi \left(  T\right)  ,\mathbb{C}1\right)  =2dist\left(
T,\mathcal{Z}\left(  \mathcal{B}\right)  \right)  ,
\]
which implies $K_{n}\left(  \mathcal{Z}\left(  \mathcal{B}\right)
,\mathcal{B}\right)  =1/2$. It is not hard to show that $\mathcal{Z}\left(
\mathcal{B}\right)  $ is metric normal when $\mathcal{B}$ has a finite
separating family of irreducible representations. However, it is also true for
$\mathcal{M}_{2}\left(  C\left(  X\right)  \right)  $ when $X$ is compact
Hausdorff space.

\item For which unital C*-algebras is every masa metric normal or metric
approximately normal? In these algebras we know that every commutative unital
C*-algebra containing the center is metric approximately normal. Morover, if,
for a centrally prime algebra $\mathcal{B}$ there is an upper bound for the
$\mathcal{K}_{an}\left(  \mathcal{A},P\mathcal{B}P\right)  $ for all masas
$\mathcal{A\subseteq}P\mathcal{B}P$ with $P$ a projection in $\mathcal{B}$,
then it follows that every $AH$ C*-subalgebra of $\mathcal{B}$ containing
$\mathcal{Z}\left(  \mathcal{B}\right)  $ is metric approximately normal.

\item It was shown in Proposition \ref{ultra} that, if each $B_{i}$ is a von
Neumann algebra, then any commutative C*-subalgebra $\mathcal{A}$ containing
the center of $%
{\displaystyle \prod_{i\in I}}
\mathcal{B}_{i}/\sum_{i\in I}\mathcal{B}_{i}$ that lifts to a commutative
C*-subalgebra of $%
{\displaystyle \prod_{i\in I}}
\mathcal{B}_{i}$, is metric approximately normal. What about those commutative
C*-algebras $\mathcal{A}$ that do not lift? We see that the general problem
almost reduces to masas that do not lift.

Interesting special cases are when $\mathcal{A}$ is the C*-algebra generated
by a single normal element or two unitary elements or three selfadjoint
elements and $I=\mathbb{N}$. It was shownn by H. Lin \cite{Lin} that when each
$\mathcal{B}_{i}$ is finite-dimensional, then every normal element in $%
{\displaystyle \prod_{i\in I}}
\mathcal{B}_{i}/\sum_{i\in I}\mathcal{B}_{i}$ lifts to a normal element in $%
{\displaystyle \prod_{i\in I}}
\mathcal{B}_{i}$. P. Friis and M. R\o rdam \cite{FR} gave a simple proof of
Lin's result when each $\mathcal{B}_{i}$ is a finite von Neumann algebra. If
$I$ is infinite and $\mathcal{B}_{i}$ is an infinite von Neumann algebra for
infinitely many $i\in I$, then there is a normal element $T$ in $%
{\displaystyle \prod_{i\in I}}
\mathcal{B}_{i}/\sum_{i\in I}\mathcal{B}_{i}$ that does not lift to a normal
element of $%
{\displaystyle \prod_{i\in I}}
\mathcal{B}_{i}$. Indeed, if $S$ is a nonunitary isometry and
\[
T_{n}=\left[  S^{n}\left(  S^{\ast}\right)  ^{n}+%
{\displaystyle \sum_{k=1}^{n}}
\frac{k}{n}S^{K}\left(  1-SS^{\ast}\right)  \left(  S^{\ast}\right)
^{k}\right]  S,
\]
then $\left \Vert T_{n}T_{n}^{\ast}-T_{n}^{\ast}T_{n}\right \Vert \leq2/n$ and
the distance from $T_{n}$ to the normal operators is $1$. Is $C^{\ast}\left(
\left \{  T\right \}  \cup \mathcal{Z}\left(
{\displaystyle \prod_{i\in I}}
\mathcal{B}_{i}/\sum_{i\in I}\mathcal{B}_{i}\right)  \right)  $ normal or
approximately normal? What is a masa in $%
{\displaystyle \prod_{i\in I}}
\mathcal{B}_{i}/\sum_{i\in I}\mathcal{B}_{i}$ that contains $T$? There is a
similar example (see \cite{D}) when $I=\mathbb{N}$ and $\mathcal{B}%
_{n}=\mathcal{M}_{n}\left(  \mathbb{C}\right)  $ for each $n$. There is a
commuting family $\left \{  T_{1},T_{2},T_{3}\right \}  $ of selfadjoint
operators in $%
{\displaystyle \prod_{i\in I}}
\mathcal{B}_{i}/\sum_{i\in I}\mathcal{B}_{i}$ that does not lift to commuting
selfadjoints in $%
{\displaystyle \prod_{i\in I}}
\mathcal{B}_{i}$. There is also \cite{V} a commuting pair $U,V$ of unitaries $%
{\displaystyle \prod_{i\in I}}
\mathcal{B}_{i}/\sum_{i\in I}\mathcal{B}_{i}$ that do not lift to commuting
unitaries in $%
{\displaystyle \prod_{i\in I}}
\mathcal{B}_{i}$. Are the associated C*-algebras genereated by these families
and the center normal or approximately normal? What are the masas in $%
{\displaystyle \prod_{i\in I}}
\mathcal{B}_{i}/\sum_{i\in I}\mathcal{B}_{i}$ in this case?

\item Let $\mathbb{F}_{3}$ denote the free group with $3$ generators $u,v,w$.
Is $C^{\ast}\left(  u,v\right)  $ approximately normal in $C^{\ast}\left(
\mathbb{F}_{3}\right)  $? In $C_{r}^{\ast}\left(  \mathbb{F}_{3}\right)  $? In
$\mathcal{L}_{\mathbb{F}_{3}}$?\bigskip
\end{enumerate}

\end{document}